\documentclass[11pt,reqno]{amsart}
 \usepackage{a4wide}
 
\numberwithin{equation}{section}
\usepackage[english]{babel}
\usepackage[T1]{fontenc}
\usepackage[latin1]{inputenc}
\usepackage{indentfirst}
\usepackage{enumitem}
\usepackage{amsmath,amssymb, amsbsy}
\usepackage{amsfonts}
\usepackage{hyperref}
\hypersetup{colorlinks,breaklinks,
            linkcolor=[rgb]{0,0,0},
            citecolor=[rgb]{0,0,0},
            urlcolor=[rgb]{0,0,0}}

\usepackage{latexsym}

\usepackage{amsthm}
\usepackage[dvips]{graphicx}
\usepackage{xcolor}
\usepackage{tikz}
\usepackage{pgfplots}
\usetikzlibrary{intersections}
\usepackage{tikz-3dplot}
\usepackage[outline]{contour}
\DeclareGraphicsExtensions{.pdf,.png,.jpg,.eps}

\usepackage{color}
\usepackage[latin1]{inputenc}
\usepackage[active]{srcltx}
\definecolor{Arancio}{cmyk}{0,0.61,0.87,0}

\definecolor{blus}{RGB}{0,102,204}

\newcommand{\brd}[1]{\mathbb{#1}}
\newcommand{\R}{\brd{R}}
\newcommand{\C}{\brd{C}}
\newcommand{\N}{\brd{N}}
\def\O{\Omega}

\newcommand{\norm}[2]{\left\Vert {#1} \right\Vert_{#2}}
\newcommand{\eps}{\varepsilon}
\newcommand{\be}{\begin{equation}}
\newcommand{\ee}{\end{equation}}
\newcommand{\loc}{{\text{\tiny{loc}}}}

\newcommand\intn{- \hspace{-0.40cm} \int}

\newtheorem{teo}{Theorem}[section]
\newtheorem{Corollary}[teo]{Corollary}
\newtheorem{Lemma}[teo]{Lemma}
\newtheorem{Theorem}[teo]{Theorem}
\newtheorem{Proposition}[teo]{Proposition}
\theoremstyle{definition}
\newtheorem{Definition}[teo]{Definition}

\newtheorem{remark}[teo]{Remark}

\pgfplotsset{compat=1.17}
\begin{document}

\title[A priori H\"older estimates for equations degenerating on nodal sets]{A priori H\"older estimates\\ for equations degenerating on nodal sets}
\thanks{{\bf Acknowledgments.}
The authors are research fellows of Istituto Nazionale di Alta Matematica INDAM group GNAMPA. G.T. and S.V. are supported by the GNAMPA project E5324001950001 \emph{PDE ellittiche che degenerano su variet\`a di dimensione bassa e frontiere libere molto sottili}. S.T. is supported by the PRIN project 20227HX33Z \emph{Pattern formation in nonlinear phenomena}. S.V. is supported by the PRIN project 2022R537CS \emph{$NO^3$ - Nodal Optimization, NOnlinear elliptic equations, NOnlocal geometric problems, with a focus on regularity}. The authors thank G. Cora and G. Fioravanti for discussions on Proposition \ref{Sobolev}.
}

\author{Susanna Terracini, Giorgio Tortone and Stefano Vita}

\address{Susanna Terracini\newline\indent
Dipartimento di Matematica "Giuseppe Peano"
\newline\indent
Universit\`a degli Studi di Torino
\newline\indent
Via Carlo Alberto 10, 10124, Torino, Italy}
\email{susanna.terracini@unito.it}
\address{Giorgio Tortone\newline\indent
Dipartimento di Matematica "Giuseppe Peano"
\newline\indent
Universit\`a degli Studi di Torino
\newline\indent
Via Carlo Alberto 10, 10124, Torino, Italy}
\email{giorgio.tortone@unito.it}
\address{Stefano Vita\newline\indent
Dipartimento di Matematica "Felice Casorati"
\newline\indent
Universit\`a di Pavia
\newline\indent
Via Ferrata 5, 27100, Pavia, Italy}
\email{stefano.vita@unipv.it}

\begin{abstract}
We prove \emph{a priori} H\"older bounds for continuous solutions to degenerate equations with variable coefficients of type
\begin{equation*}
\mathrm{div}\left(u^2 A\nabla w\right)=0\quad\mathrm{in \ }\Omega\subset\R^n,\qquad \mbox{with}\qquad \mathrm{div}\left(A\nabla u\right)=0,
\end{equation*}
where $A$ is a Lipschitz continuous, uniformly elliptic matrix (possibly $u$ has non-trivial singular nodal set). Such estimates are uniform with respect to $u$ in a class of normalized solutions that have a bounded Almgren frequency. 
As a consequence, a boundary Harnack principle holds for the quotient of two solutions vanishing on a common set. 

This analysis relies on a detailed study of the associated weighted Sobolev spaces, including integrability of the weight, capacitary properties of the nodal set, and uniform Sobolev inequalities yielding local boundedness of solutions.
\end{abstract}
\maketitle

\section{Introduction}

The \emph{boundary Harnack principle on nodal domains} is a comparison principle between pairs of solutions $u,v$ to the same uniformly elliptic PDE
\begin{equation}\label{equv}
\mathrm{div}\left(A\nabla u\right)=\frac{\partial}{\partial x_i}\left(a_{ij}\frac{\partial}{\partial x_j}u\right)=0\qquad\mathrm{in \ }B_1,
\end{equation}
satisfying a geometric inclusion of their zero sets, precisely 
$$
Z(u)\subseteq Z(v),\quad\mbox{where}\quad Z(u):=u^{-1}\{0\}\cap B_1.
$$ 
This comparison can be formulated  
in terms of boundedness, H\"older continuity, or even Schauder-type regularity of their ratio $v/u$, depending on the regularity of the coefficients in \eqref{equv}. Although the classical theory of the boundary Harnack principle is well understood on given domains, even under mild regularity assumptions on their boundaries (e.g., \cite{Kem72,Dal77,Anc78,JerKen,BanBasBur91,BasBur91}), it is only in recent years that different research groups have developed a fine regularity theory in the context of nodal domains.

First, in \cite{LogMal1,LogMal2} Logunov and Malinnikova proved real-analyticity of the quotient $v/u$ of two harmonic functions across the nodal set $Z(u)$. Their analysis extends to solutions $u,v$ of \eqref{equv} involving analytic coefficient matrices. 

Under mild regularity assumptions, Lin and Lin \cite{LinLin} addressed the case of Lipschitz-continuous coefficients and established $\alpha^*$-H\"older continuity of the quotient, for some (implicitly defined) small exponent $\alpha^*\in(0,1)$, depending on the Almgren frequency of the denominator $u$. The key tools of their analysis rely on the validity of the corkscrew condition for this class of nodal domains and the existence of modified Harnack-chains, based on the Almgren monotonicity formula.

In \cite{TerTorVit1}, we propose a different approach based on the regularity theory of degenerate elliptic PDEs. As already remarked in \cite{LogMal1}, the ratio $w:=v/u$ solves an elliptic equation in divergence form whose coefficients degenerate along the nodal set $Z(u)$
\begin{equation}\label{eqw}
\mathrm{div}\left(|u|^aA\nabla w\right)=0\qquad\mathrm{in \ } B_1,
\end{equation}
with $a=2$. Note that the degeneracy of the operator is influenced by the exponent $a \in \R$, the vanishing order of $u$, and the geometry of its nodal set.\\

The main intention of this paper is to provide a comprehensive overview of the functional framework associated with the degenerate PDE in \eqref{eqw}, in the presence of non-trivial nodal sets, Lipschitz continuous coefficients and general exponents $a\in\R$. Ultimately, we establish \emph{a priori} $\alpha$-H\"older bounds for continuous solutions in the case $a = 2$, for every $\alpha \in (0,1)$. These estimates are uniform with respect to $u$, and thus with respect to its nodal set $Z(u)$, within the compact class $\mathcal{S}_{N_0}$ of solutions with bounded Almgren frequency (see \eqref{classS.intro}). 

Higher-order regularity results, such as Schauder estimates and a posteriori bounds, are the subject of our companion paper \cite{TerTorVit3}, where the problem is studied for general powers $a\in\R$ of the weight.

\subsection{Main results}
In order to state our results, some definitions are in order. Through the paper, we consider coefficients $A \in \R^{n\times n}$ that are symmetric, Lipschitz continuous and uniformly elliptic; that is,
\begin{equation}\label{UE}
\lambda|\xi|^2\leq A(x)\xi \cdot \xi \leq\Lambda|\xi|^2 \qquad \mbox{in }\, B_1, \ \mbox{for any } \, \xi \in \R^n,
\end{equation}
for some $0<\lambda\leq\Lambda<+\infty$. Then, by the standard Schauder theory, weak solutions are $C^{1,1-}_\loc(B_1)$ (i.e. $C^{1,\alpha}_\loc(B_1)$ for any $\alpha\in(0,1)$) and the nodal set $Z(u)$ splits into a regular part $R(u)$ and a singular part $S(u)$ defined as
$$
R(u)=\{x\in Z(u) \ : \ |\nabla u(x)|\neq0\},\qquad S(u)=\{x\in Z(u) \ : \ |\nabla u(x)|=0\}.
$$
$R(u)$ is locally the graph of $C^{1,1-}$ functions and $S(u)$ has Hausdorff dimension at most $(n-2)$ (see e.g. \cite{Han,HanLin2,GarLin} for further details).

Given $0<\lambda\leq\Lambda$, $L>0$, the class of admissible coefficients $\mathcal A=\mathcal A_{\lambda,\Lambda,L}$ is given by
\begin{equation}\label{coeff.ipotesi}
\begin{aligned}
\mathcal A &=\left\{
A:=(a_{ij})_{ij} \, : \, A=A^T, \, A(0)=\mathbb I,\, \mathrm{satisfies} \,\eqref{UE}\mbox{ and } [a_{ij}]_{C^{0,1}(B_1)}\leq L\right\}.
\end{aligned}
\end{equation}
Then, given $N_0>0$, the class $\mathcal S_{N_0}:=\mathcal S_{N_0,\lambda,\Lambda,L}$ of solutions $u$ of the elliptic equation \eqref{equv}, whose powers are the weight terms in \eqref{eqw}, is given by
\begin{equation}\label{classS.intro}
\mathcal S_{N_0}:=\left\{u \in H^1(B_1) \, : \, u\mbox{ solves }\eqref{equv}
 \mbox{ with } \,0 \in Z(u), A\in\mathcal A, \, \|u\|_{L^2(B_{1/2})}=1, \,N(0,u,1)\leq N_0\right\}
\end{equation}
where $N$ is the extended Almgren frequency associated with \eqref{equv} (see Section \ref{structure.nodal} for a precise definition of $N$).
The choice of this class is natural in order to achieve uniformity within it. Indeed, the validity of the Almgren monotonicity formula allows us to estimate the size of nodal and critical sets, as well as the vanishing orders of solutions, in terms of bounds on their frequency (see also \cite{LinLin,NabVal,Han,HanLin2}).\\

To start with, the uniform-in-$\mathcal{S}_{N_0}$ local boundedness of solutions to \eqref{eqw} follows from a Moser iteration technique (see Proposition \ref{prop.bound}).
However, the latter requires the validity of weighted Sobolev inequalities which are uniform-in-$\mathcal{S}_{N_0}$ (see Proposition \ref{Sobolev}). This already implies the classical boundary Harnack principle on nodal domains, i.e. the local boundedness of the ratio $w=v/u$, with uniformity within the class $\mathcal S_{N_0}$ (see also Remark \ref{r:giorgiopoi}).

Then, our main result concerns uniform-in-$\mathcal{S}_{N_0}$ \emph{a priori} H\"older bounds for the ratio $w=v/u$ of solutions to \eqref{equv} with $u\in\mathcal S_{N_0}$ and $Z(u)\subseteq Z(v)$ in any space dimension. As already mentioned, this amounts to consider equation \eqref{eqw} with $a=2$.\\
For the sake of readability, throughout the paper we will say that a constant $C>0$ \emph{depends on the class $\mathcal{S}_{N_0}$}, if it only depends on $n, N_0,\lambda,\Lambda$ and $L$.

\begin{Theorem}[\emph{A priori} uniform-in-$\mathcal S_{N_0}$ H\"older bounds for the ratio]\label{uniformHolderZ}
Let $n\geq2$, $A\in\mathcal A, u\in \mathcal{S}_{N_0}$ and $v$ be solutions to
$$
\begin{cases}
\mathrm{div}\left(A\nabla u\right) = 0 &\mbox{in }B_1,\\
\mathrm{div}\left(A\nabla v\right) = 0 &\mbox{in }B_1
\end{cases}\quad\mbox{satisfying}\quad Z(u)\subseteq Z(v).
$$
Then, if $v/u\in C^{0,\alpha}_{\loc}(B_1)$ there exists $C>0$, depending only on $\mathcal{S}_{N_0}$, and $\alpha\in(0,1)$ such that
\begin{equation*}
\left\|\frac{v}{u}\right\|_{C^{0,\alpha}(B_{1/2})}\leq 
C\|v\|_{L^2(B_1)}.
\end{equation*}
\end{Theorem}

We wish to highlight some relevant features of the main result above.
    \begin{enumerate}
        \item[(i)] In \cite{LogMal2}, Logunov and Malinnikova prove the validity of uniform $C^{k,\alpha}$-estimates (for any order $k$, in two dimensions) for quotients of harmonic functions sharing zero sets.
        In their paper, they exploited a compactness result and a Liouville theorem both by Nadirashvili \cite{Nad0,Nad} for a class of planar harmonic functions with a uniform control over the number of nodal domains (see \eqref{classT.intro} for the precise definition of this class).  On the other hand, in case of Lipschitz continuous coefficients and in any dimension, the $\alpha^*$-H\"{o}lder estimate obtained by Lin and Lin in \cite{LinLin} is uniform-in-$\mathcal{S}_{N_0}$.

        At this point, it is interesting to compare the two uniformity classes. We show that in two dimensions the class considered in \cite{LogMal2} coincides with $\mathcal S_{N_0}$ (in the sense of Propositions \ref{SinT} and \ref{TinS}). Ultimately, in the case of  harmonic polynomials, we provide an algebraic expression which relates the number of nodal domains with the different vanishing orders at singular points in $\mathbb{R}^2$ (see Proposition \ref{propNad}).

        \item[(ii)] Obviously, Theorem \ref{uniformHolderZ} does assume the regularity of solutions \emph{a priori}. However, any time this regularity is \emph{de facto}, i.e, effective, these results imply, as a byproduct, uniform-in-$\mathcal S_{N_0}$ \emph{a posteriori} regularity estimates as in \cite{LogMal2,LinLin}. For instance, Theorem \ref{uniformHolderZ} allows to extend to any space dimension the uniform H\"older estimates obtained by \cite{LogMal2} in case of real analytic coefficients (see Corollary \ref{corLogMal}).
        \end{enumerate}
        
In line with the program developed in \cite{SirTerVit1, SirTerVit2, TerTorVit1}, which extends the classical strategy in \cite{Sim} to the case of degenerate equations, the H\"{o}lder estimates in the theorem above are obtained through a fine blow-up analysis combined with Liouville-type theorems, that is, classification results for global solutions to
\begin{equation}\label{e:a-entire}
\mathrm{div}(u^2\nabla w)=0\qquad\mathrm{in} \ \R^n,
\end{equation}
where $u$ is a homogeneous harmonic polynomial and $w$ exhibits polynomial growth. The following result recasts a division lemma due to Murdoch \cite{Mur} as a Liouville-type theorem for pairs of solutions with shared zero sets, where such rigidity can be directly observed at the level of entire solutions.

\begin{Theorem}[Liouville theorem for the ratio of harmonic functions sharing zero sets]\label{liouvilletheorem}
Let $n\geq2$, $u$ be a harmonic polynomial and $v$ be an entire harmonic function in $\R^n$ such that $Z(u)\subseteq Z(v)$. Suppose that there exist $\gamma\geq0$ and $C>0$ such that
  \begin{equation}\label{growth.eq}
  \left|\frac{v}{u}\right|(x)\leq C(1+|x|)^{\gamma}\quad\mbox{in }\R^n.
  \end{equation}
  Then the ratio $v/u$ is a polynomial of degree at most $\lfloor \gamma\rfloor$.
\end{Theorem}

We remark that
Theorem \ref{liouvilletheorem} extends the Liouville theorem \cite[Corollary 4.7]{LinLin} in which the existence of a least growth at infinity is deduced by a Harnack type inequality for the ratio.

\subsection{Structure of the paper}
In Section \ref{sec:class} we highlight the main features of the uniformity class $\mathcal{S}_{N_0}$, based on the validity of the Almgren monotonicity formula, and we compare this class with the one in \cite{LogMal2}. Then in Section \ref{sec:functional}, we give the precise notion of solution to PDEs whose coefficients degenerate on nodal sets, we discuss the functional framework giving a close look to the natural weighted Sobolev spaces associated to \eqref{eqw}, with a particular attention to local integrability properties of the weight terms and capacitary features of their nodal sets.

Moreover, for solutions to \eqref{eqw}, we establish local $L^\infty$ bounds that are uniform-in-$\mathcal S_{N_0}$, based on Sobolev inequalities and Moser iterations.

Section \ref{sec:holder} is devoted to the proof of the main Theorem \ref{uniformHolderZ}; that is, the \emph{a priori} uniform-in-$\mathcal S_{N_0}$ H\"older estimates for the ratio of two solutions sharing zero sets.

Finally, in Section \ref{sec:liouville}, we prove the Liouville Theorem \ref{liouvilletheorem}. While studying the behavior of a planar harmonic polynomial at infinity, we show the validity of an algebraic formula which relates the number of its nodal domains with the number of its singular points counted with multiplicity (vanishing orders). Then Theorem \ref{liouvilletheorem} relies on the validity of a division lemma between polynomials by Murdoch.

\section{The compact class of solutions \texorpdfstring{$\mathcal{S}_{N_0}$}{Lg}}\label{sec:class}
In this section we introduce several known results related to solutions of elliptic PDEs with Lipschitz variable coefficients. Starting from the validity of an Almgren monotonicity formula, we show the known properties of the class of solutions $\mathcal{S}_{N_0}$.

\subsection{Almgren monotonicity formula and structure of the nodal set}\label{structure.nodal}
Let $n\geq 2$ and $u \in H^1(B_1)$ be a weak solution to \eqref{equv}
where $A \in \mathcal{A}$ is a symmetric uniformly elliptic matrix with Lipschitz continuous coefficients (see \eqref{coeff.ipotesi}). In such case, by elliptic regularity any weak solution is of class $C^{1,1-}_\loc(B_1) \cap H^2_\loc(B_1)$.
 Thus, by \cite{Han} (see also \cite{HanLin2,GarLin}) the nodal set $Z(u)=u^{-1}\{0\}$ of $u$ splits into a regular part $R(u)$, which is locally a $(n-1)$-dimensional hyper-surface of class $C^{1,1-}$, and the singular part $S(u)$ which has Hausdorff dimension at most $(n-2)$.

We introduce now the notion of vanishing order.

\begin{Definition}[Vanishing order]
Given $u \in H^1(B_1)$, we can define the \emph{vanishing order} of $u$ at $x_0 \in B_1$ as
$$
\mathcal{V}(x_0,u) = \sup\left\{k\geq 0: \ \limsup_{r \to 0^+} \frac1{r^{n-1+2k}} \int_{\partial B_r(x_0)} u^2 \,\mathrm{d}\sigma <+\infty\right\}.
$$
\end{Definition}

The number $\mathcal{V}(x_0,u) \in [0,+\infty]$ is characterized by the property that
$$
\limsup_{r \to 0^+} \frac1{r^{n-1+2k}} \int_{\partial B_r(x_0)} u^2  \,\mathrm{d}\sigma= \begin{cases} 0 & \text{if $0 <k< \mathcal{V}(x_0,u)$} \\
+\infty  & \text{if $k > \mathcal{V}(x_0,u)$}.
\end{cases}
$$
The Lipschitz continuity of the coefficients of $A$ allows to prove the strong unique continuation principle, see \cite{GarLin}, which consists in the fact that non-trivial solutions can not vanish with infinite order at $Z(u)$. Ultimately, it implies that the only solution vanishing in an open subset of its reference domain is the trivial one, which is the classical unique continuation principle.

The Lipschitz continuity of the leading coefficients allows to derive the existence of a monotonicity formula, which is a key tool to the local analysis of solutions near their nodal set. For the sake of completeness, we recall the construction of the generalized Almgren frequency function, by following the general ideas in \cite{GarLin,HanLin2,Han} (see also the recent developments in \cite{CheNabVal,NabVal}). The key idea is to consider a specific change of coordinates in order to rewrite the problem in terms of a Laplace-Beltrami operator associated to a new given metric $g$. We suggest to the interested reader to address \cite[Section 3]{CheNabVal} for more details on this construction.

Thus, by following \cite{CheNabVal}, fixed $x_0\in B_1$ we consider the function
$$
d^2(x_0,x)=a^{ij}(x_0)(x-x_0)_i(x-x_0)_j,
$$
where $x=x_i e_i$ is the canonical decomposition in $\R^n$ and $a^{ij}$ are the entries of $A^{-1}$. Then, consider the new metric $g(x)=g(x_0,x)$ centered at $x_0$ defined by
\begin{equation}\label{e:g}
g_{ij}(x_0,x)=\eta(x_0,x) a^{ij}(x)\quad\mbox{with}\quad
\eta(x_0,x)=\frac{a_{kl}(z)a^{ks}(x_0)a^{lt}(x_0)}{d^2(x_0,x)}(x-x_0)_s(x-x_0)_t.
\end{equation}
Notice that the geodesic distance in the metric $g(x)$ centered at $x_0$ corresponds to $d(x_0,x)$, for every choice of $x_0,x \in B_1$. Thus, we can defined with $E_\rho(x_0)$ the geodesic ball centered at $x_0 \in B_1$, with radius $r>0$, induced by $g_{x_0}$ as 
$$
E_r(x_0) := x_0 + r A^{-1/2}(x_0)(B_1).
$$

Before giving the definition of the Almgren-type frequency function, we can rewrite equation \eqref{equv} in terms of the Laplace-Beltrami operator associated to the metric $g$. Indeed, by a direct computation, $u$ is a solution to \eqref{equv} if and only if it solves
\begin{equation}\label{e:NabVal}
\Delta_g u + \left(\nabla_g \log(|g|^{-1/2}\eta)\cdot \nabla_g u\right)_g =0,
\end{equation}
where $|g|$ is the determinant of the metric $g$ and
$$
\Delta_g u =\frac{1}{\sqrt{|g|}}\partial_i\left( \sqrt{|g|} g^{ij}\partial_j u\right)\quad\mbox{and}\quad \nabla_g u = g^{ij}\partial_j u e_i.
$$
It is more convenient to rewrite \eqref{e:NabVal} in terms of a weighted divergence PDEs, that is
$$
\mathrm{div}_g(\omega \nabla_g u) =0,\qquad\mbox{with }\quad\omega= |g|^{-1/2}\eta(x_0,x).
$$
For every $x_0 \in B_1, r \in (0,\Lambda^{-1/2}(1-|x_0|))$, set
$$
  D(x_0,u,g,r) := \int_{E_r(x_0)} \omega |\nabla_g u|^2_{g}  \,\mathrm{d} V_{g},\qquad
H(x_0,u,g,r) := \int_{\partial E_r(x_0)} \omega u^2\,\mathrm{d}\sigma_g,
$$
\begin{remark}
In \cite{GarLin} the authors considered a rather similar metric $\overline{g}=\overline{g}(x_0,x)$ defined as $\overline{g}_{ij}(x_0,x) = g_{ij}(x_0,x) (\mathrm{det}\, A(x))^{\frac{1}{n-2}}$, with $g$ as in \eqref{e:g}. With respect to $\overline{g}$, it can be proved that $u$ is a solution to \eqref{equv} if and only if it solves
$$
\mathrm{div}_{\overline{g}}(\mu \nabla_{\overline{g}} u) =\frac{1}{\sqrt{|\overline{g}|}}\partial_i\left(\mu \sqrt{|\overline{g}|}\overline{g}^{ij}\partial_j u\right) =0,\quad\mbox{with }\mu = \eta^{\frac{2-n}{2}},
$$
and $\eta$ defined in \eqref{e:g}. We stress that the two metrics are equivalent and they allow to construct frequency functions that coincide among matrices $A\in \mathcal{A}$ with constant coefficients.
\end{remark}
Throughout the paper, if no ambiguity occurs, we will often omit the dependence on the metric \eqref{e:g} in the frequency function. By adapting the results of \cite{GarLin, CheNabVal,NabVal, HanLin2} to our notations, we can state the following monotonicity result.
\begin{Proposition}[\cite{GarLin,CheNabVal,HanLin2}]\label{prop.monoton}
Let $u$ be a solution to \eqref{equv} in $B_1$ and $x_0 \in B_1, r\in (0,\Lambda^{-1/2}(1-|x_0|))$. Then there exists $C>0$ depending only on $n,\lambda,\Lambda$ and $L$, such that the map
\be\label{e:Almgren}
N(x_0,u,g,\cdot) \colon r \mapsto r e^{C r}\frac{D(x_0,u,g,r)}{H(x_0,u,g,r)}
\ee
is monotone non-decreasing for  $r< \Lambda^{-1/2}(1-|x_0|)$. Moreover, for $0<r_1<r_2<\Lambda^{-1/2}(1-|x_0|)$ it holds
that
$$
\left|\frac{H(x_0,u,g,r_1)}{H(x_0,u,g,r_2)}\exp\left(-2\int_{r_1}^{r_2} \frac{N(x_0,u,g,\tau)}{t}\,\mathrm{d}\tau\right)-1\right|\leq C r_2.
$$
\end{Proposition}
Then, we define as generalized Almgren frequency, the formula in \eqref{e:Almgren}. Notice that, if $A \in \mathcal{A}$ has constant coefficients, then we can simply take $C=0$ in \eqref{e:Almgren}.\\
First, by the monotonicity, for every $x_0 \in B_1$ we deduce the existence of the limit
$$
N(x_0,u,g,0^+)=\lim_{r\to 0^+}N(x_0,u,g,r),
$$
which coincides with the vanishing order of $u$ at $x_0$, that is
$
\mathcal{V}(x_0,u)= N(x_0,u,g,0^+).
$
Secondly, Proposition \ref{prop.monoton} implies the following doubling-type estimate: by \cite[Theorem 1.3]{GarLin} (see also \cite[Section 3.1]{HanLin2}) there exists $C=C(n,\lambda,\Lambda,L)$ such that, for every $x_0 \in B_{1/2}$, we have
\begin{equation}\label{tipo.doubling}
\int_{B_{R}(x_0)}u^2 \,\mathrm{d}x\leq C\left(\frac{R}{r}\right)^{n + 2N(x_0,u,g,R)}\int_{B_{r}(x_0)}u^2 \,\mathrm{d}x,
\end{equation}
for every $0<r\leq R\leq \Lambda^{-1/2}(1-|x_0|)$.\\
Despite the presence of the metric $g$ in \eqref{e:Almgren}, which deeply depends on the choice of the center $x_0 \in B_1$, the generalized frequency formula enjoys invariance properties similar to those which hold for the Almgren formula for harmonic functions along blow-up sequences.
We introduce here a quantity that will be used through the paper several times: in light of the monotonicity formula, there exists $\overline{C}>0$ depending only on $n,\lambda, \Lambda$ and $L$ such that
$$
N(x_0,u,r)\leq \overline{C} N(0,u,1)\leq \overline{C} N_0, \quad\mbox{for every }x_0 \in B_{7/8}, r\leq 1/16.
$$
Thus, we define
\begin{equation}\label{overlineN0}
    \overline N_0:=\sup_{u\in \mathcal{S}_{N_0}}\max_{x_0\in \overline{B_{7/8}}, r\leq 1/16} N(x_0,u,r) \leq  \overline{C}N_0.
\end{equation}
In addition, through this definition, it is possible to bound the vanishing orders of $u$ on $Z(u)\cap B_{7/8}$ uniformly-in-$\mathcal{S}_{N_0}$.
\begin{Definition}[Blow-up sequences]\label{d:blowup}
Let $(A_k)_k \subset \mathcal{A}$ and $(u_k)_k \subset H^1(B_1)$ be a family of solution to $L_{A_k}u_k=0$ in $B_1$. Then, fixed two sequences $(x_k)_k \subset B_{3/4}$ and $r_k\searrow 0^+$ with $0<r_k< \Lambda^{-1/2}/4$, we define $g_k$ the metric introduced in \eqref{e:g} centered at $x_k$ and depending on $A_k$. Finally, we define as blow-up sequence centered at $x_k$ the family
\begin{equation}\label{e:blow-up}
U_k(x)= \frac{u_k\left(x_k+r_k x\right)}{H(x_k,u_k,g_k,r_k)^{1/2}},\quad\mbox{for }x \in B_{1/r_k},
\end{equation}
satisfying $L_{\tilde{A}_k}U_k=0$ in $B_{1/r_k}$, with $\tilde{A}_k(0)=\mathbb I$ and such
$$
\tilde{A}_k(x) := A^{-1/2}(x_k) A_k (x_k+ r_k A^{-1/2}(x_k)x) A^{-1/2}(x_k)\in \mathcal{A}.
$$
Ultimately, by exploiting the definition of \eqref{e:g} and \eqref{e:Almgren}, it holds
\begin{equation}\label{e:Almgren.change}
N(x_k,u_k,g_k,r_k r) = N(0,U_k,\tilde{g}_k,r),\quad\mbox{for }0<r<\frac{\Lambda^{-1/2}}{4 r_k},
\end{equation}
with $\tilde{g}_k$ the metric centered at the origin associated to the matrix $\tilde{A}_k$.
\end{Definition}
By exploiting the compactness in $\mathcal{A}$ and the one of solutions to \eqref{equv} with bounded $L^2$-norm, it is possible to prove existence of blow-up limits. Indeed, for blow-up sequences with fixed center $x_k=x_0 \in B_1$, it is possible to deduce existence and uniqueness of the blow-up limit by following the ideas in \cite[Theorem 1.5]{Han} (see also \cite[Section 3]{Han}). Generally, it is more convenient to prove compactness of the blow-up sequence among functions with uniformly bounded frequency (see also \cite[Section 2.2]{NabVal} and \cite[Section 3]{CheNabVal} for a interesting discussion on the notion of blow-ups for more general elliptic equations).\\

Thus, in view of the definition of $\mathcal{A} =\mathcal A_{\lambda,\Lambda,L}$ in \eqref{coeff.ipotesi}, given $N_0>0$ we define the class of solutions $\mathcal S_{N_0}=\mathcal S_{\lambda,\Lambda,L,N_0}$ as
$$
\mathcal S_{N_0}=\left\{u \in H^1(B_1) \, : \, u\mbox{ solves }\eqref{equv}
 \mbox{ with }  A\in\mathcal A, \,0 \in Z(u),\, N(0,u,g,1)\leq N_0, \, \|u\|_{L^2(B_{1/2})}=1\right\}.
$$
Heuristically, in view of the monotonicity result and the doubling condition, a bound on $N(0,u,g,1)$ also gives a uniform control on $\mathcal{V}(x,u)$ and $N(x,u,g_r)$, for well-chosen $x\in B_1, r>0$ (see for instance \cite[Theorem 3.2.10]{HanLin2} and \cite[Lemma 3.9]{CheNabVal}).

Although the compactness of the class $\mathcal S_{N_0}$ in the $C^{1,\alpha}\cap H^1$ topology is well-known (see for instance \cite{Han,HanLin2}), in the following result we want to highlight the general compactness result for blow-up sequences in $\mathcal S_{N_0}$ with moving centers, which will be crucial for the contradiction arguments in the following sections. Indeed, the uniform bound $N_0$ ultimately implies that blow-up limits are harmonic polynomials.

\begin{Proposition}[Compactness of blow-up sequences]\label{p:blow-up}
Let $(u_k) \subset \mathcal S_{N_0}$ be a family of solutions associated to $(A_k)_k \subset \mathcal{A}$. Le us consider the blow-up sequence $U_k$ centered at $(x_k)_k \subset B_{3/4}$ with $r_k\searrow 0^+$ given in Definition \ref{d:blowup}. Then, there exists a harmonic polynomial $P$, such that, up to a subsequence, $U_k \to P$ in $C^{1,\alpha}_\loc(\R^n)$, for every $\alpha \in (0,1)$ and strongly in $H^1_\loc(\R^n)$. Moreover, $Z(U_k)\to Z(P)$ with respect to the Hausdorff distance.
\end{Proposition}
\begin{proof}
The proof of the result follows by known arguments of the blow-up analysis carried out in \cite{Han,HanLin2,CheNabVal,NabVal}. Indeed, since the blow-up sequence
$$
\tilde{A}_k(x):=A^{-1/2}(x_k) A_k (x_k+ r_k A^{-1/2}(x_k)x) A^{-1/2}(x_k) \to \mathbb I,
$$
uniformly on every compact set of $\R^n$, by standard arguments we deduce that the blow-up limit is harmonic in $\R^n$.
Indeed, by exploiting the doubling condition \eqref{tipo.doubling}, it is possible to prove a uniform bound in $H^1_\loc(\R^n)$ and to deduce then convergence in $C^{1,\alpha}_\loc(\R^n)\cap H^1_\loc(\R^n)$. Moreover, the Hausdorff convergence of the nodal sets of $U_k$ follows by the maximum principle and the unique continuation principle for harmonic polynomials and the uniform convergence of the sequence $(U_k)$ on every compact set of $\R^n$.\\
Then, by \cite[Theorem 3.2.10]{HanLin2}, since $\tilde{g}_k \to \delta_{ij}$ locally and uniformly in $\R^n$, by \eqref{e:Almgren.change} we have
\begin{equation}\label{boundsullaN}
N(0,P,\delta_{ij},t) = \lim_{k\to +\infty}N(x_k,U_k,\tilde{g}_k,r_k t) \leq C(n,\lambda,\Lambda,L)N_0 \quad\mbox{for every }t>0,
\end{equation}
where $N(0,P,\delta_{ij},t)$ is the classical frequency for harmonic functions in $\R^n$. Finally, considering the limit $t\to+\infty$ in \eqref{boundsullaN}, in view of Proposition \ref{lem.natural} we can conclude that $P$ is a polynomial of degree $N_\infty$ smaller or equal than $ C(n,\lambda,\Lambda,L) N_0$.
On one side, by the $C^{1,\alpha}_\loc(\R^n)$-convergence, we immediately deduce that
for every $\varepsilon>0$ there exists $\bar k>0$ such that
$$
Z(U_k) \cap B_1 \subseteq N_\varepsilon (Z(P)\cap B_1), \quad\mbox{for }k\geq \bar k,
$$
where $N_\varepsilon(\cdot)$ is the closed $\varepsilon$-neighborhood of a set. On the other hand, let us show that for every $\varepsilon>0$ there exists $\bar k>0$ such that
$$
Z(P) \cap B_1 \subseteq N_\varepsilon (Z(U_k)\cap B_1), \quad\mbox{for }k\geq \bar k.
$$
First, given $x_0\in Z(P), \delta>0$ there exists $k>0$ large such that $Z(U_k)\cap B_\delta(x_0)\neq \emptyset$. If not, up to a change of sign, we would have $U_k>0$ in $B_\delta(x_0)$ for every $k>0$ sufficiently large. Then, by passing to the limit, we deduce that $P\geq 0$ in $B_\delta(x_0)$ and $P(x_0)=0$, which implies by the maximum principle that $P\equiv 0$ in $B_\delta(x_0)$, in contradiction with the unique continuation principle.

Now, suppose by contradiction the existence of $\bar\varepsilon>0$ and $x_k \in Z(P)\cap B_1, x_k\to x \in Z(P)\cap \overline{B_1}$ such that $\mathrm{dist}(x_k,Z(U_k)\cap B_1)\geq \bar\varepsilon$. Since $Z(P)$ is homogeneous and passes through the origin, there exists $\bar x \in Z(P)\cap B_1$ such that $|x-\bar x|\leq \bar \varepsilon/4$. Moreover, by making use of the result proved in the previous paragraph, we can take a sequence $(\bar{x}_k)_k \subset Z(U_k)\cap B_1$ such that $|\bar{x}_k-\bar x|\leq \bar \varepsilon/4$, for large $k>0$. Then
$$
\mathrm{dist}(x_k, Z(U_k)\cap B_1)\leq
|x_k-\bar x_k|\leq |x_k-x|+|x-\bar{x}|+|\bar{x}_k-\bar{x}|< \bar\varepsilon
$$
for large $k$, which is a contradiction. Finally, the result can be extended on every compact set of $\R^n$ by scaling.
\end{proof}

As we have already pointed out, the bound of the frequency in the class $\mathcal S_{N_0}$ implies a uniform bound on the vanishing order at well-chosen points. In the following Lemma we rephrase such principle in terms of a uniform non-degeneracy type result.
\begin{Lemma}\label{lemma.hay}
Let $u \in \mathcal{S}_{N_0}$ and $\overline{N}_0$ be defined as in \eqref{overlineN0}, then for every $x_0 \in B_{3/4}$ and $0<r\leq 1/8$, we have
$$
\int_{B_{r}(x_0)}u^2 \,\mathrm{d}x\geq C r^{n+2\overline{N}_0}\int_{B_1}u^2\,\mathrm{d}x
$$
where $C$ depends only on $\lambda,\Lambda,L$ and $N_0$.
\end{Lemma}
\begin{proof}
For the sake of simplicity, we omit the dependence of the Almgren frequency on the metric $g$. By the definition of $\overline N_0$ (see also \cite[Theorem 3.2.10]{HanLin2}), we have
$$
N\left(x_0,u,r\right) \leq \overline N_0
,\qquad \mbox{for every }x_0 \in B_{3/4}, r \leq 1/8.
$$
Now, we split the proof in two cases.\\

Case 1: $x_0 \in B_{1/16}$. Then, by \eqref{tipo.doubling}, there exists $C=C(\lambda,\Lambda,L)$ such that
$$
\int_{B_{r}(x_0)}u^2 \,\mathrm{d}x\geq C8^{n+2\overline{N}_0} r^{n+2\overline{N}_0}\int_{B_{1/8}(x_0)}u^2\,\mathrm{d}x,
$$
for every $r \leq 1/8$. On the other hand, since $x_0\in B_{1/8}$ and $r<1/16$, we have that $B_{1/16}\subset B_{1/8}(x_0)$ and so
\begin{align*}
\int_{B_{r}(x_0)}u^2 \,\mathrm{d}x
\geq C 8^{n+2\overline{N}_0}r^{n+2\overline{N}_0}\int_{B_{1/16}}u^2\,\mathrm{d}x
\geq C \left(\frac12\right)^{n+2\overline{N}_0} r^{n+2\overline{N}_0}\int_{B_1}u^2\,\mathrm{d}x
\end{align*}
where in the last inequality we apply \eqref{tipo.doubling} at $x_0=0$ with $r=1/16$ and $R=1$.\\

Case 2: $x_0 \in B_{3/4}\setminus B_{1/16}$. In this case, we just need to iterate the previous argument, in order to reduce to Case 1. Given $x_0 \in B_{3/4}\setminus B_{1/16}, \, r<1/8$ we set $\rho_0=1/8$ and, for $i\in \N$, we consider the sequences
$$
x_{i+1} = \frac54 x_i - \frac14 \frac{x_i}{|x_i|},\qquad \rho_{i+1} = \frac58(1-|x_i|).
$$
With this choice, we get $B_{\rho_{i}/2}(x_{i+1})\subset B_{\rho_i}(x_i)$ and so
$$
\int_{B_{r}(x_0)}u^2\,\mathrm{d}x \geq C \frac{r^{n+2\overline{N}_0}}{\rho_0^{n+2\overline{N}_0} }\int_{B_{\rho_0}(x_0)}u^2 \,\mathrm{d}x\geq
C \frac{r^{n+2\overline{N}_0}}{\rho_0^{n+2\overline{N}_0} }\int_{B_{\frac{\rho_1}{2}}(x_1)}u^2 \,\mathrm{d}x\geq C \frac{r^{n+2\overline{N}_0}}{(\rho_0 \rho_1)^{n+2\overline{N}_0} }\int_{B_{\rho_1}(x_1)}u^2\,\mathrm{d}x.
$$
Since $|x_{i+1}|<|x_i|$ and $\rho_{i+1}>\rho_i$, by repeating this argument a finite number of times (with our choice the number of iteration is always less than $6$), we finally get
$$
\int_{B_{r}(x_0)}u^2\,\mathrm{d}x \geq \tilde{C} \frac{r^{n+2\overline{N}_0}}{(\rho_0 \rho_1)^{n+2\overline N_0}}\int_{B_{\overline{\rho}}(\overline{x})}u^2\,\mathrm{d}x,
$$
with $\overline{\rho}<1/16$ and $\overline{x} \in B_{1/8}$, and the result follows by applying Case 1.
\end{proof}
\begin{remark}
In \cite{LinLin,GarLin} the authors show similar growth estimate for the solution $u$ in terms of the distance from its nodal set. More precisely, by combining the $C^{1,\alpha}$-regularity of solutions to \eqref{equv} with the doubling-type estimate of Lemma \ref{lemma.hay}, there exists $C_2= C_2(n,\lambda,\Lambda)$ and $C_1$ depending only on the class $\mathcal{S}_{N_0}$, such that given $u \in \mathcal{S}_{N_0}$ and $\overline N_0$ as in \eqref{overlineN0}, we have
$$
C_1\mathrm{dist}(x,Z(u))^{\overline N_0}  \leq |u(x)| \leq C_2 \mathrm{dist}(x,Z(u)),
$$
for $x \in B_{3/4}$ (see \cite[Theorem 2.5]{LinLin}). Among the various implications of these inequalities, we stress the following estimate which relates the behavior of the solutions in $\mathcal{S}_{N_0}$ in two different neighborhoods of the nodal set: let $x_0\in Z(U)$ and $x_1\in \{u>0\}\cap B_{R}(x_0)$, then for $r<R<1/8$ we have
$$
\intn_{B_R(x_0)}u^2\,\mathrm{d}x \leq C \frac{1}{(R-r)^{\overline N_0}}\intn_{B_r(x_1)}u^2\,\mathrm{d}x
$$
for some $C$ depending only on the class $\mathcal{S}_{N_0}$.
\end{remark}

By using the class $\mathcal S_{N_0}$ it is possible to provide upper bounds for the size of the nodal sets and the singular one. Indeed, by \cite{HarSim,CheNabVal, NabVal} there exists $C_1,C_2>0$ depending on $\mathcal{S}_{N_0}$, such that
$$
\mathcal{H}^{n-1}(Z(u)\cap B_{1/2})\leq C_1,\quad
\mathcal{H}^{n-2}(S(u)\cap B_{1/2})\leq C_2,\qquad\mbox{for every }u\in\mathcal{S}_{N_0}.
$$
It is worth highlighting that in \cite{CheNabVal, NabVal} the authors established estimates of the constants $C_1,C_2$ (see also \cite{Lin} for the conjectured sharp dependence on $N_0$), while in \cite{CheNabVal, NabVal} estimates on the Minkowski content of $Z(u), S(u)$ and the set of critical points as well are obtained. Notice that in \cite{Han2}, the author provide the optimal dependence of the measure of $S(u)$ on $N_0$ in the $2$-dimensional case.
\subsection{Further results in two dimensions} 
In the two dimensional case $n=2$, it is possible to extend the analysis of the nodal set of solutions to elliptic PDEs to equations whose coefficients are H\"{o}lder continuous.
In such case, by elliptic regularity any weak solution is of class $C^{1,\alpha}_\loc(B_1)$ and,
        $$
        u(z) = P_{z_0}(A^{1/2}(z_0)(z-z_0)) + \Gamma_{z_0}(x) \quad \mbox{in }B_r(z_0)
        $$
        with
        $$
        |\Gamma_{z_0}(z)| \leq C|z-z_0|^{\mathcal{V}(z_0,u)+\delta},\quad|\nabla\Gamma_{z_0}(z)| \leq C|z-z_0|^{\mathcal{V}(z_0,u)-1+\delta}
        \qquad\quad\mbox{in }B_r(z_0),
        $$
        for some constants $C,\delta>0$ depending only on $\lambda, \Lambda$ and the H\"{o}lder seminorm of $|\nabla u|$. Thus, the nodal set $Z(u)$ splits into a regular part $R(u)$, which is locally a $(n-1)$-dimensional hyper-surface of class $C^{1,\alpha}$, and the singular part $S(u)$ which is a locally finite collection of isolated points. We refer to \cite{Aless1} for a detailed discussion on the topic.
\subsection{Comparison between two compact classes of solutions}
Let us introduce the uniformity compact class considered in \cite{LogMal2}; that is, a class of solutions in dimension two with a bound on the number of nodal domains. Given $n=2$ and $k_0\in\mathbb N\setminus\{0\}$, let
\begin{equation}\label{classT.intro}
\mathcal T_{k_0}=\left\{u \in H^1(B_1) \, : \, u\mbox{ solves }\eqref{equv}
 \mbox{ with }  A\in\mathcal A, \, C(u,B_1) \leq k_0, \, \|u\|_{L^2(B_{1/2})}=1\right\},
\end{equation}
where
$$C(u,B_1)= \#\left\{\mbox{connected components of } \{u\neq0\}\cap B_1\right\}.$$
In this section we show that $\mathcal S_{N_0}=\mathcal T_{k_0}$ when $n=2$ and $A$ is the identity matrix, in the sense of the following results (see also \cite[Corollary 3.3]{LinLin} for a $n$-dimensional proof of Proposition \ref{SinT}).
\begin{Proposition}[$\mathcal S_{N_0}\subseteq\mathcal T_{k_0}$]\label{SinT}
Let $n=2$ and $N_0>0$. Then there exist $0<\theta_0<1/2$ and $k_0\in\mathbb{N}\setminus\{0\}$ such that for any $0<r<\theta_0$,
$$\inf_{\rho\in[\frac{r}{2},\frac{3}{2}r]}C(u,B_\rho)\leq k_0$$
for any $u\in\mathcal S_{N_0}$.
\end{Proposition}
\proof
Seeking a contradiction, for some $N_0>0$ and along the sequences $\theta_n=\frac{1}{k}$, $u_k\in\mathcal S_{N_0}$ there exists $r_k<\theta_k$ such that
\begin{equation}\label{Cu_n}
\inf_{\rho\in[\frac{1}{2}r_k,\frac{3}{2}r_k]}C(u_k,B_\rho)\geq k.
\end{equation}
The blow-up sequence $U_k$ centered at the origin and associated to the sequence $r_k$ (see Definition \ref{d:blowup}) converges in $C^{1,\alpha}$ on any compact set to the limit $U$ which is a harmonic polynomial in $\R^2$ with $\mathrm{deg}(U)\leq N_0$. Let $\overline\rho\in[\frac{1}{2},\frac{3}{2}]$ chosen so that the nodal lines of $U$ intersect transversally $\partial B_{\overline\rho}$ and no singular point of $U$ lies on it. Then $U$ has a finite number of nodal domains in $B_{\overline\rho}$ and this should hold also for $U_k$ by $C^{1,\alpha}$ convergence, in contradiction with \eqref{Cu_n}.
\endproof

In order to prove the other inclusion, at least in case the coefficients are constant, we need a compactness result for harmonic functions in the class $\mathcal T_{k_0}$ by Nadirashvili. The result is \cite[Lemma 2.1]{LogMal2} obtained combining arguments in \cite{Nad0,Nad}.

\begin{Lemma}[Nadirashvili's compactness theorem]\label{NadCompact}
    Let $u_k$ be a sequence of harmonic functions in $B_1$ having at most $k_0$ nodal regions. Then there exist a subsequence $u_{k_i}$, a sequence $\alpha_{k_i}$ of real numbers and a nontrivial function $u$ such that $\alpha_{k_i}u_{k_i}$ converge to $u$ uniformly on compact subsets of $B_1$. Moreover, $u$ is harmonic.
\end{Lemma}

Then, at least in case $A=\mathbb I$, one can prove the other inclusion.
\begin{Proposition}[$\mathcal T_{k_0}\subseteq\mathcal S_{N_0}$]\label{TinS}
Let $n=2$, $A=\mathbb I$, $k_0>0$ and $0<\theta<1$. Then there exists $N_0>0$ such that
$$N(0,u,\theta)\leq N_0$$
for any $u\in\mathcal T_{k_0}$.
\end{Proposition}
\proof
Let us suppose by contradiction the existence of $k_0>0$ and $0<\theta<1$ such that along a sequence $u_k$ in $\mathcal T_{k_0}$
\begin{equation}\label{Ntheta>n}
N(0,u_k,\theta)\geq k.
\end{equation}
Then, by Lemma \ref{NadCompact}, up to subsequences there exist real numbers $\alpha_k$ such that $\alpha_k u_k$ converge to $u$ uniformly on compact sets. The latter convergence implies $H^1_\loc$ convergence. Then
\begin{equation*}
H(0,\alpha_k u_k,\theta)\to H(0,u,\theta)>0,\qquad
        E(0,\alpha_k u_k,\theta)\to E(0,u,\theta)<+\infty.
\end{equation*}
Then $N(0,\alpha_k u_k,\theta)\to N(0,u,\theta)\in[0,+\infty)$, in contradiction with \eqref{Ntheta>n}.
\endproof

\section{Elliptic PDEs degenerating on nodal sets}\label{sec:functional}
In this section we give the precise notion of solutions to equations \eqref{eqw} for general exponents; that is,
\begin{equation*}
\mathrm{div}\left(|u|^a A\nabla w\right)=0\qquad \mbox{and}\qquad \mathrm{div}\left(A\nabla u\right)=0\quad\mathrm{in \ }B_1.
\end{equation*}
The main results of the section include a classification of exponents $a \in \R$ for which the weight is integrable, a comprehensive characterization of weighted Sobolev spaces (with special attention to the validity of the property $(H\equiv W)$). Moreover, we discuss in more detail the peculiar case $a=2$ and its connection with ratios of $A$-harmonic functions, i.e. solutions to \eqref{equv}, sharing zero sets. Finally, we provide Sobolev inequalities and local boundedness for solutions to \eqref{eqw}, with constants uniform-in-$\mathcal{S}_{N_0}$.

\subsection{Weighted Sobolev spaces}
We start by showing how the presence of singular points affects the local integrability of the weight $|u|^a$.
To maintain conciseness in notations, in this work we consider solutions $u \in \mathcal{S}_{N_0}$, and the estimates are formulated uniformly with respect to this class. It is worth noting that, for a given solution $u$ of \eqref{equv}, these results can be also expressed by substituting the dependence on $N_0$ with the Almgren frequency \eqref{e:Almgren} evaluated in $B_1$.

Let us define
\begin{equation}\label{a_S}
    a_{\mathcal{S}}:=\min\left\{1,\frac{2}{\overline N_0}\right\}\in(0,1],
\end{equation}
where $\overline N_0$ is the quantity introduced in \eqref{overlineN0}.
\begin{Lemma}\label{l:a-N0}
    Let $n\geq 2, u\in \mathcal{S}_{N_0}$ and $a \in \R$. Then
    \begin{itemize}
        \item[\rm{(i)}] $|u|^a \in L^1(B_{7/8})$ for every $a>-a_{\mathcal{S}},$ uniformly-in-$\mathcal{S}_{N_0}$;
        \item[\rm{(ii)}] $|u|^a$ is $\mathcal{A}_2$-Muckenhoupt in $B_{7/8}$ for every $a \in (-a_{\mathcal{S}},a_{\mathcal{S}})$, uniformly-in-$\mathcal{S}_{N_0}$.
    \end{itemize}
\end{Lemma}
\begin{proof}

First, in view of the analysis developed in \cite{Han,Han2,HanLin2}, for every $x_0 \in Z(u)\cap B_{7/8}$ and $r\in (0,1-|x_0|)$, there exist a non-trivial homogeneous harmonic polynomial $P_{x_0}$ of degree $\mathcal{V}(x_0,u)$ and a function $\Gamma_{x_0}$ such that
        \be\label{taylor}
        u(x) = P_{x_0}(A^{1/2}(x_0)(x-x_0)) + \Gamma_{x_0}(x) \quad \mbox{in }B_r(x_0)
        \ee
        with
        $$
        \begin{cases}
        |\Gamma_{x_0}(x)| \leq C|x-x_0|^{\mathcal{V}(x_0,u)+\delta}\\
        |\nabla\Gamma_{x_0}(x)| \leq C|x-x_0|^{\mathcal{V}(x_0,u)-1+\delta}
        \end{cases}\quad\mbox{in }B_r(x_0)
        $$
        for some constants $C,\delta>0$ depending only on $n, \lambda, \Lambda$ and $L$.
We stress that, by definition of \eqref{overlineN0}, the vanishing order $\mathcal{V}(x_0,u)$ are bounded by $\overline{N}_0$, for every $x_0 \in B_{7/8}$ and uniformly-in-$\mathcal{S}_{N_0}$.

        Notice also that the local integrability of $|u|^a$ in $B_{7/8}$, follows once we prove local integrability of harmonic polynomials $P$ of degree $k\leq \overline N_0$.
        In other words, we prove the result showing that, whenever the weight has the form $|\omega|^a$, where $\omega$ satisfies a local expansion of the form \eqref{taylor} with $P$ being a homogeneous harmonic polynomials of degree $k\leq \overline N_0$, then it is locally integrable for $a>-a_{\mathcal{S}}$.

        Now, we proceed by proving the latter statement by induction on the dimension $n\geq 2$. In order to ease the exposition, with a small abuse of notation, we will set any integration in the unitary ball even if one should do it in the ball $B_{7/8}$ where the vanishing order of functions $u\in\mathcal{S}_{N_0}$ is bounded uniformly-in-$\mathcal{S}_{N_0}$ by \eqref{overlineN0}. \\

\noindent Case $n=2$. In view of the preliminary discussion, let $P$ be a homogeneous harmonic polynomial in $\R^2$ and  $(r,\theta) \in [0,+\infty)\times [0,2\pi)$ be the polar coordinates in $\R^2$. Then, up to a rotation and a multiplicative constant, we have that
$$
\quad P(r,\theta) = r^k \cos(k\theta)\quad\mbox{with }k\geq 1.
$$
We stress that $P$ depends on two-variables if and only if $k\geq 2$. It is clear that if $k=1$ then $P$ is linear (i.e.  one-dimensional), then $|P|^a \in L^1_\loc(\R^2)$ if and only if $a>-1$. On the other hand, if $k\geq 2$ (i.e. two dimensional) we have
$$
\int_{B_1} |P(x)|^a \,\mathrm{d}x = \int_0^1 r^{ka+1}\,\mathrm{d}r\int_0^{2\pi}|\cos(k\theta)|^a\,\mathrm{d}\theta < +\infty
$$
if and only if $a>-2/k\geq -1$. Now, the result follows by applying the assumption that $k\leq  \overline N_0$.\\

\noindent Case $n\geq 3$. Up to a rotation, we can suppose that the $k$-homogeneous harmonic polynomial $P$ either is $d$-dimensional with $1\leq d\leq n-1$, that is
$$
P(x',x'')=P(x',0), \quad \mbox{for every }(x',x'') \in \R^{d}\times \R^{n-d}
$$
or it depends on $n$ variables and $k\geq n$. The case $P$ is $d$-dimensional with $1\leq d\leq n-1$ follows from the inductive step: indeed if we consider the generalized cylinder $C_1 = B_1'\times [-1,1]^{n-d}\subset \R^n$, with $B_1'\subset \R^d$ the unitary $d$-dimensional ball, we get
$$
\int_{C_1} |P(x)|^a \,\mathrm{d}x = \int_{B_1'} |P(x',0)|^a \,\mathrm{d}x',
$$
which is finite for every $a>-a_{\mathcal{S}}$. Therefore, we can restrict ourselves to the case where $P$ is of degree $k\geq n$ and it depends on $n$-variables. More precisely, in view of the dimension reduction of homogeneous functions and the upper semi-continuity of the vanishing order, we can assume that
$$
\{x \in \R^n\colon \mathcal{V}(x,P)=k\} = \{0\}\qquad \mbox{and}\qquad
\mathcal{V}(x_0,P)\leq k-1,\mbox{ for every }x_0 \in Z(P).
$$
Therefore, if we denote with $(r,\theta) \in [0,+\infty)\times \mathbb S^{n-1}$ the spherical coordinates in $\R^n$, we get
\be\label{e:Pndim}
\int_{B_1} |P(x)|^a \,\mathrm{d}x = \int_0^1 r^{ka+n-1}\,\mathrm{d}r\int_{\mathbb S^{n-1}}|P(\theta)|^a\,\mathrm{d}\sigma(\theta),
\ee
where $P\colon \mathbb S^{n-1} \to \R$ satisfies
\be\label{e:Pn-equation}
-\Delta_{\mathbb S^{n-1}} P = k (k+n-2) P\quad \mbox{in }\mathbb S^{n-1},\qquad
\mathcal{V}(x_0,P) \leq k-1 \quad \mbox{for every }x_0 \in Z(P)\cap \mathbb S^{n-1}.
\ee
Naturally, the first integral in the right hand side of \eqref{e:Pndim} is finite if and only if $a>-n/k$.\\
On the other hand, up to rephrase the Taylor-type expansion \eqref{taylor} with respect to the geodesic distance of $\mathbb S^{n-1}$, we know that solutions to \eqref{e:Pn-equation} enjoy an expansion of the form \eqref{taylor} in dimension $n-1$. Thus, by the inductive step, the second integral in \eqref{e:Pndim} is finite for every
$$
a> -a_{\mathcal{S}}.
$$
and the result follows immediately once we notice that $a_{\mathcal{S}} < n/\overline N_0$. On the other hand, the final part of the result follows by requiring that both $|u|^a, |u|^{-a}$ are locally integrable.
\end{proof}

Now, we delve into the notion of weighted Sobolev spaces, with a focus on the interplay between sets of zero weighted capacity and the class of regular functions dense in the energy space.
\begin{Definition}[Weighted Sobolev space $H$]\label{d:sobolev}
Let $u\in\mathcal{S}_{N_0}$ and $a>-a_{\mathcal{S}}$, then the weighted Sobolev-type space $H^1(B_1,|u|^a)$ is defined as the completion of $C^\infty(\overline{B_1})$ with respect to the norm
\be\label{e:norm-ua}
\norm{w}{H^1(B_1,|u|^a)}^2 = \int_{B_1} |u|^a \left(w^2 + |\nabla w|^2\right)\,\mathrm{d}x.
\ee
\end{Definition}
\begin{Lemma}\label{l:capacity}
Let $u\in\mathcal{S}_{N_0}$. Then:
\begin{enumerate}
\item[\rm{(i)}] if $a \geq 1$, then the nodal set $Z(u)\cap B_1$ has null $H^1(|u|^a)$-capacity;
\item[\rm{(ii)}] if $a \geq 0$, then the singular set $S(u)\cap B_1$ has null $H^1(|u|^a)$-capacity.
\end{enumerate}
\end{Lemma}
\begin{proof}
    We split the proof into two cases accordingly to the values of $a$.\\

    Case $a\geq 1$. Miming the construction in \cite[Remark 3.3]{TerTorVit1}, let $\eps>0$ and consider
$$
f_\eps(t)=\begin{cases}
    1 &\mbox{if } |t|\leq \eps^2\\
    \frac{\log|t|- \log\eps}{\log\eps} &\mbox{if } \eps^2\leq |t|\leq \eps \\
    0 &\mbox{if } |t|\geq \eps.
\end{cases}
$$
Therefore, if we set $w_\eps:=f_\eps(u) \in H^1(B_1)$, by coarea formula, we get
$$
\begin{aligned}
\int_{B_1}|u|^a|\nabla w_\eps|^2\,\mathrm{d}x = \frac{1}{(\log\eps)^2}\int_{B_1\cap \{\eps^2\leq |u|\leq \eps\}}|u|^{a-2}|\nabla u|^2\,\mathrm{d}x \leq \frac{C}{(\log\eps)^2}\norm{\nabla u}{L^\infty(B_1)}^2\int_{\eps^2}^\eps t^{a-2}\,\mathrm{d}t
\end{aligned}
$$
where, by integration, the right hand side approaches zero as $\eps \to 0^+$. Therefore, up to regularize $w_\eps$ with mollifiers with sufficiently small support, we deduce that the nodal set $Z(u)\cap B_1$ has null $H^1(|u|^a)$-capacity.\\

Case $a\geq 0$. Since the $H^1$-capacity of the singular set $S(u)$ is null, the result follows immediately once we notice that
\be\label{e:confronto-misure}
\int_{B_1}|u|^a \varphi^2 \,\mathrm{d}x \leq \norm{u}{L^\infty(B_1)}^a
\int_{B_1}\varphi^2 \,\mathrm{d}x,\qquad
\int_{B_1}|u|^a |\nabla \varphi|^2 \,\mathrm{d}x \leq \norm{u}{L^\infty(B_1)}^a
\int_{B_1}|\nabla \varphi|^2 \,\mathrm{d}x,
\ee
for every $\varphi \in H^1(B_1)$.
\end{proof}
\begin{Proposition}\label{p:alternative-H}
Let $u\in\mathcal{S}_{N_0}$. Then
\begin{enumerate}
    \item[\rm{(i)}] if $a\geq 1$, then $H^1(B_1, |u|^a)$ can be defined as the completion of $C^\infty_c(\overline{B_1}\setminus Z(u))$ with respect to the norm \eqref{e:norm-ua};
    \item[\rm{(ii)}] if $a \in [0,1)$, then $H^1(B_1, |u|^a)$ can be defined as the completion of $C^\infty_c(\overline{B_1}\setminus S(u))$ with respect to the norm \eqref{e:norm-ua}.
\end{enumerate}
More generally, if $E\subset\R^d$ is a set with zero-Lebesgue measure and null $H^1(|u|^a)$-capacity, then the class $C^\infty_c(\overline{B_1}\setminus E)$ is dense in $H^1(B_1,|u|^a)$ with respect to the norm \eqref{e:norm-ua}.
\end{Proposition}

\begin{proof}
We stress that already in \cite[Remark 3.3]{TerTorVit1} (see also see Lemma 4.2 and Theorem 4.7 in \cite{Zhikov}) the authors remarked that if the weighted capacity of sets where the weight strongly degenerates (specifically $Z(u)$ if $a\geq 1$ and $S(u)$ if $a\in [0,1)$) is null, then it is possible to characterize $H^1(B_1,|u|^a)$ as the completion of a restricted class of smooth functions.\\
Indeed, in view of Lemma \ref{l:capacity}, for every $\eps>0$ there exists $\phi_\eps \in C^\infty_c(B_1)$ such that $0\leq \phi_\eps \leq 1, \phi_\eps \to 0$ a.e. in $B_1$, $\norm{|\nabla \phi_\eps|}{L^2(B_1,|u|^a)}\leq \eps$ and
$$
\mbox{if }a\geq 1,\mbox{ then }\,\phi_\eps\equiv 1\,\,\mbox{in }N_\eps(Z(u))\cap B_1;\qquad
\mbox{if }a\in [0,1),\mbox{ then }\,\phi_\eps\equiv 1\,\,\mbox{in }N_\eps(S(u))\cap B_1.
$$
Now, given $w \in C^\infty(\overline{B_1})$, the approximating sequence $w_\eps(x) := w(x)(1-\phi_\eps(x))$ strongly converges to $w$ with respect to the norm \eqref{e:norm-ua}. Indeed
\be\label{e:strong-convergence-HW}
\begin{aligned}
\int_{B_1}|u|^a (w_\eps - w)^2\,\mathrm{d}x
&\leq \int_{B_1 } |u|^a w^2 \phi_\eps^2 \,\mathrm{d}x \to 0^+,\\
\int_{B_1}|u|^a|\nabla (w_\eps- w)|^2\,\mathrm{d}x &\leq \int_{B_1 } |u|^a |\nabla w|^2 \phi_\eps^2 \,\mathrm{d}x + \int_{B_1 }|u|^{a} w^2|\nabla \phi_\eps|^2\,\mathrm{d}x\\
&\leq \int_{B_1}|u|^a |\nabla w|^2 \phi_\eps^2 \,\mathrm{d}x +  \norm{u}{L^\infty(B_1)}^a \norm{w}{L^\infty(B_1)}^2 \eps^2 \to 0^+,
\end{aligned}
\ee
where in the last inequality we used the definition of $\phi_\eps$ and the limit follows by dominated convergence. Ultimately, the result follows since $w_\eps$ is identically zero in a neighborhood of $Z(u)$ (resp. in a neighborhood of $S(u)$) in $B_1$ if $a\geq 1$ (resp. $a \in [0,1)$).
\end{proof}

We conclude this subsection by showing the validity of the $(H\equiv W)$-property for the weighted Sobolev spaces of Definition \ref{d:sobolev}; that is, the equivalence between the space of functions with finite well-defined norm and the completion of smooth functions with respect to the norm.
\begin{Proposition}[The $(H\equiv W)$-property]
Let $u\in\mathcal{S}_{N_0}, a>-a_{\mathcal{S}}$ and
\begin{enumerate}
    \item[\rm{(i)}] if $a\geq 1$, set  $
W^{1,2}(B_1,|u|^a) := \left\{w \in W^{1,1}_\loc(B_1\setminus Z(u))\colon \norm{w}{H^1(B_1,|u|^a)}<+\infty\right\};$\vspace{0.1cm}
\item[\rm{(ii)}] if $a \in [a_{\mathcal{S}},1)$, set
$W^{1,2}(B_1,|u|^a) := \left\{w \in W^{1,1}_\loc(B_1\setminus S(u))\colon \norm{w}{H^1(B_1,|u|^a)}<+\infty\right\};$\vspace{0.1cm}
\item[\rm{(iii)}] if $a\in (-a_{\mathcal{S}},a_{\mathcal{S}})$, set $
W^{1,2}(B_1,|u|^a) := \left\{w \in W^{1,1}_\loc(B_1)\colon \norm{w}{H^1(B_1,|u|^a)}<+\infty\right\}.$\vspace{0.05cm}
\end{enumerate}
Then, $W^{1,2}(B_1,|u|^a) \equiv H^1(B_1,|u|^a)$.
\end{Proposition}
\begin{proof}
In light of Lemma \ref{l:a-N0}, in the case (iii), the function $|u|^a$ is an $\mathcal{A}_2$-Muckenhoupt weight. Therefore, in this range the validity of the property $(H=W)$ follows from classical results.\\
It remains to address the case $a\geq a_{\mathcal{S}}$. We start by stressing that for $a \geq a_{\mathcal{S}}$ the Sobolev space $W^{1,2}(B_1,|u|^a)$ is well-defined. Indeed, given $w \in W^{1,2}(B_1,|u|^a)$, since in both the cases (i)-(ii) the removed set $E=Z(u)$ for $a\geq 1$ (resp. $E=S(u)$ for $a \in [0,1)$) is a closed set of null measure, the weak derivatives are those functions $\partial_{x_i} w \in L^1_\loc(B_1\setminus E)$ satisfying
$$
\int_{B_1}|u|^a w\partial_{x_i} \varphi \,\mathrm{d}x =
-\int_{B_1} |u|^a \varphi \partial_{x_i} w\,\mathrm{d}x,\qquad \mbox{for every }\varphi \in C^\infty_c(\overline{B_1}\setminus E),
$$
where we stress again that the removed set $E$ has null Lebesgue measure and null weighted capacity. By Definition \ref{d:sobolev}, since
$$
H^1(B_1,|u|^a) \subseteq \overline{W^{1,2}(B_1,|u|^a)}^{\norm{\cdot}{H^1(|u|^a)}}
$$
the result follows once we prove that $W^{1,2}(B_1,|u|^a)$ is complete (i.e. a Banach space) with respect to the norm \eqref{e:norm-ua} and every function in $W^{1,2}(B_1,|u|^a)$ can be approximated by regular functions in $H^1(B_1,|u|^a)$ with respect to the norm \eqref{e:norm-ua}; that is,
$$W^{1,2}(B_1,|u|^a)=\overline{W^{1,2}(B_1,|u|^a)}^{\norm{\cdot}{H^1(|u|^a)}}$$
and
$$W^{1,2}(B_1,|u|^a)\subseteq H^{1}(B_1,|u|^a).$$

Step 1: $W^{1,2}(B_1,|u|^a)$ is a Banach space. Since $L^2(B_1,|u|^a)$ is a Banach space, we first get that every limit of a Cauchy sequence in $W^{1,2}(B_1,|u|^a)$ has finite norm \eqref{e:norm-ua}. On the other hand, we have two possibilities according to the value of $a \geq a_{\mathcal{S}}$:
\begin{enumerate}
    \item[\rm{(i)}] if $a\geq 1$, on every compact set $K\subset \subset B_1\setminus Z(u)$ the weight $|u|^a$ is bounded from below and above, thus every limit of a Cauchy sequence with finite norm automatically belongs to $W^{1,1}_\loc(B_1\setminus Z(u))$;
    \item[\rm{(ii)}] if $a \in [a_{\mathcal{S}},1)$, on every compact set $K\subset\subset B_1\setminus S(u)$ the weight $|u|^a$ is $\mathcal{A}_2$-Muckenhoupt and so, by Cauchy-Schwarz inequality
    $$
    \int_{K}|w| + |\nabla w|\,\mathrm{d}x \leq \left(\int_K |u|^{-a}\,\mathrm{d}x\right)^{1/2}\left(\int_{K}|u|^a (|w|^2 + |\nabla w|^2)\,\mathrm{d}x\right)^{1/2}<+\infty.
    $$
    Hence, every limit of a Cauchy sequence with finite norm belongs to $W^{1,1}_\loc(B_1\setminus S(u))$.\\
\end{enumerate}

Step 2: $W^{1,2}(B_1,|u|^a)\subseteq H^1(B_1,|u|^a)$. First, by Lemma \ref{l:capacity}, for every $\eps>0$ there exists $\phi_\eps \in C^\infty_c(B_1)$ such that $0\leq \phi_\eps \leq 1, \phi_\eps \to 0$ a.e. in $B_1$, $\norm{|\nabla \phi_\eps|}{L^2(B_1,|u|^a)}\leq \eps$ and
$$
\phi_\eps\equiv 1\,\,\mbox{in }N_\eps(Z(u))\cap B_1 \ \mbox{if }a\geq 1;\qquad
\phi_\eps\equiv 1\,\,\mbox{in }N_\eps(S(u))\cap B_1 \ \mbox{if }a\geq 0.
$$
Now, given $w \in W^{1,2}(B_1,|u|^a)$, the approximating sequence $w_\eps(x) := w(x)(1-\phi_\eps(x))$ strongly converges to $w$ with respect to the norm \eqref{e:norm-ua} (see \eqref{e:strong-convergence-HW}). We conclude by distinguishing two cases:
\begin{enumerate}
    \item[\rm{(i)}] if $a\geq 1$, since $|u|^a$ is bounded from below and above in $B_1\setminus N_\eps(Z(u))$, we deduce that $w_\eps$ belongs to $H^1(B_1)$ and then it is approximated by functions $\tilde{w}_\eps \in C^\infty(\overline{B_1})$ with respect to the $H^1$-norm. Ultimately, by \eqref{e:confronto-misure}, as $\eps \to 0^+$ it implies
    $$
    \norm{w-\tilde{w}_\eps}{H^1(B_1,|u|^a)}^2\leq \norm{w-w_\eps}{H^1(B_1,|u|^a)}^2+
    \norm{u}{L^\infty(B_1)}^a \norm{w_\eps-\tilde{w}_\eps}{H^1(B_1)}^2\to 0;
    $$
    \item[\rm{(ii)}] if $a \in [a_{\mathcal{S}},1)$, for a fixed $\eps>0$ consider the $\mathcal{A}_2$-Muckenhoupt weight
    $$
    \rho_\eps(x) := \begin{cases}
    \norm{u}{L^\infty(B_1)}^a &\mbox{in }N_\eps(S(u))\cap B_1\\
    |u(x)|^a &\mbox{in }B_1\setminus N_\eps(S(u)).
    \end{cases}
    $$
    Then, $w_\eps\in H^1(B_1,\rho_\eps dx)$ and so, in light of the $(H\equiv W)$-property for weighted Sobolev spaces involving $\mathcal A_2$-weights, there exists  $\tilde{w}_\eps \in C^\infty(\overline{B_1})$ such that
    $$
    \norm{w_\eps -\tilde{w}_\eps}{H^1(B_1,\rho_\eps dx)}^2 \leq \eps^2,
    $$
    where the weighted norm is the natural one associated to $\rho_\eps$. Finally, by combining the previous estimates, as $\eps\to 0^+$ we have
    $$
    \norm{w-\tilde{w}_\eps}{H^1(B_1,|u|^a)}^2\leq \norm{w-w_\eps}{H^1(B_1,|u|^a)}^2+\norm{w_\eps-\tilde{w}_\eps}{H^1(B_1,\rho)}^2\to 0.
    $$
    \end{enumerate}
\end{proof}
\begin{remark}
In the case $a \in (-1,-a_{\mathcal{S}})$, the presence of singular points strongly affects the structure of the Sobolev space $H^1(B_1, |u|^a)$. Indeed, by miming the program developed for the case of $u$ linear and $a\leq -1$ (see \cite{SirTerVit2}), we can prove that the finiteness of the norm \eqref{e:norm-ua} implies that functions in the energy space - and possibly their higher order derivatives, depending on the vanishing order of $u$ at the singularity - must vanish on the singularity.
\end{remark}

\subsection{The notion of solution for weighted PDEs} We introduce the class of solutions considered through the section.
\begin{Definition}\label{definition.energy.a}
Let $A \in \mathcal{A}$, $u \in \mathcal{S}_{N_0}$ and $a > -a_{\mathcal{S}}$ (with $a_{\mathcal{S}}$ as in \eqref{a_S}). Then, we distinguish two cases:
\begin{enumerate}
    \item[\rm{(i)}] if $a\geq 1$, then we say that $w\in H^1(B_1,|u|^a)$ is a solution to \eqref{eqw} in $B_1$ if
\begin{equation*}
\int_{B_1}|u|^a A\nabla w\cdot\nabla\phi\,\mathrm{d}x=0,\quad \text{for every } \phi\in C^\infty_c(B_1)\,;
\end{equation*}
\item[\rm{(ii)}] if $a \in (-a_{\mathcal{S}},1)$, then we say that $w\in H^1(B_1,|u|^a)$ is a solution to \eqref{eqw} in $B_1$ satisfying
\be\label{e:equation-exponent-a-neumann}
|u|^a  A \nabla w \cdot \nabla u = 0\quad\mbox{on }R(u)\cap B_1,
\ee
if, on every connected component $\O_u$ of $\{u\neq 0\}$, we have
$$
\int_{\O_u \cap B_1}|u|^a A\nabla w\cdot\nabla\phi\,\mathrm{d}x=0,\quad \text{for every } \phi\in C^\infty_c(B_1).
$$
\end{enumerate}
Naturally, it is equivalent to state the weak formulations with respect to the dense classes of smooth functions defined in  Proposition \ref{p:alternative-H}.
\end{Definition}

In the second case of Definition \ref{definition.energy.a}, we are restricting the attention to those problems in which the conormal boundary condition \eqref{e:equation-exponent-a-neumann} holds true. Indeed, it is possible to construct  solutions to \eqref{eqw} satisfying a homogeneous Dirichlet condition on $R(u)\cap B_1$ that are not differentiable close to the regular set $R(u)$ (see \cite{SirTerVit2} for the case $u$ linear).

\begin{remark}
    We would like to remark that in the superdegenerate case $a\geq1$, the two formulations (i), (ii) in the previous definition are equivalent. Indeed, the test functions for the weak formulation can be taken in $C^\infty_c(B_1\setminus Z(u))$, and this fact implies that the equation is satisfied in particular on any nodal domain of $u$. However, in this case, since the $H^1(|u|^a)$-capacity of the nodal set $Z(u)$ is zero, the boundary condition in \eqref{e:equation-exponent-a-neumann} does not have sense in general but is somehow formally always satisfied.
\end{remark}

\subsection{The case \texorpdfstring{$a=2$}{Lg} and the ratio of solutions of PDEs sharing nodal sets}
We present some specific results for the case $a=2$. Naturally, being $|u|^2 \in L^\infty(B_1)$, the condition on the local integrability of the weight is satisfied. As highlighted in \cite[Proposition 3.5]{TerTorVit1}, the ratio of solutions of the equation \eqref{equv} satisfies \eqref{eqw} in the sense of Definition \ref{definition.energy.a}.

In the following result, we observe that there is, indeed, an equivalence between the two problems: it is possible to demonstrate that the product of a solution $w$ of \eqref{eqw} and its weight $u$ is an $A$-harmonic function (see \eqref{equv}).

\begin{Proposition}[Inverse relation for the ratio]\label{p:inverse}
Let $u\in H^1(B_1)$ and $w\in H^1(B_1,u^2)$ be weak solutions respectively to \eqref{equv} and \eqref{eqw} in $B_1$. Assume that $w \in L^\infty(B_1) \cap C(B_1\setminus S(u))$, then the function $v=uw \in H^1(B_1)$ solves \eqref{equv} in $B_1$ with $Z(u)\subseteq Z(v)$.
\end{Proposition}
\begin{proof}
It is immediate to notice that $v\in H^1(B_1)$ and $Z(u)\subseteq Z(v)$, indeed
$$
\int_{B_1} v^2 + |\nabla v|^2 \,\mathrm{d}x \leq 2\left(\int_{B_1} u^2 w^2 + u^2|\nabla w|^2 + w^2|\nabla u|^2 \,\mathrm{d}x\right),
$$
where the right hand side is finite since $w\in H^1(B_1,u^2)\cap L^\infty(B_1)$ and $u \in H^1(B_1)$. Moreover, by assumption we already know that $v \in C(B_1\setminus Z(u))$. Now, since for every $\phi \in C^\infty_c(B_1\setminus Z(u))$ it holds
$$
\int_{B_1} A\nabla v \cdot \nabla \phi\,\mathrm{d}x = \int_{B_1} u^2 A\nabla w \cdot \nabla \left(\frac{\phi}{u}\right)\,\mathrm{d}x=0,
$$
we have that $\mathrm{div}\left(A\nabla v\right)=0$ in $B_1\setminus Z(u)$. Clearly, the previous equality is obtained by applying an approximation argument on $\phi/u \in C^1(B_1\setminus Z(u))$ (see for example \cite[Remark 3.2]{TerTorVit1}). In the remaining part of the proof we first extend the validity of the equation across $R(u)$ and ultimately across $S(u)$ by exploiting a capacity argument.\\
Let $x_0 \in R(u)$ and $r_0>0$ be such that $B_{r_0}(x_0) \cap Z(u) = B_{r_0}(x_0)\cap R(u)$ and $B_{r_0}(x_0)\setminus R(u) = \Omega_+\cup \Omega_-$ where $\partial\Omega_{\pm}\cap B_{r_0}(x_0)$ is the graph of a $C^{1,\alpha}$-function and $v$ is a weak-solution of $\mathrm{div}\left(A\nabla v\right) = 0$ in $\Omega_+\cup \Omega_-$. Now, set $N_\eps(R(u))$ to be the $\eps$-tubular neighborhood of $R(u)$, then for every $\phi \in C^\infty_c(B_{r_0}(x_0))$
$$
\int_{B_{r_0}(x_0)\setminus N_\eps(R(u))}A \nabla v \cdot \nabla \phi\,\mathrm{d}x =
\int_{\Omega_+\cap \partial N_\eps(R(u))} \phi A \nabla v \cdot \nu_{+,\eps}\,\mathrm{d}\sigma + \int_{\Omega_-\cap\partial N_\eps(R(u))}\phi A\nabla v \cdot \nu_{-,\eps}\,\mathrm{d}\sigma
$$
where, up to a rotation, we have $\nu_{+,\eps} = \nabla u/|\nabla u| + o(1)$ and $\nu_{-,\eps} = -\nabla u/|\nabla u| + o(1)$. Since $w$ is a weak solution to $\eqref{eqw}$, we get that
$$
\left|\int_{\Omega_\pm\cap \partial N_\eps(R(u))} \phi A \nabla v \cdot \nu_{\pm,\eps} - w A \nabla u \cdot \nu_{\pm,\eps} \,\mathrm{d}\sigma\right| \leq
\int_{\Omega_\pm\cap\partial N_\eps(R(u))} \phi u A \nabla w \cdot \nu_{\pm,\eps}\,\mathrm{d}\sigma
$$
Thus, in view of \cite[Theorem 1.3]{TerTorVit1}, since $A\nabla w \cdot \nu = 0$ on $R(u)\cap B_1$ we get
$$
\int_{B_{r_0}(x_0)\setminus N_\eps(R(u))}A \nabla v \cdot \nabla \phi \,\mathrm{d}x= o(1)
$$
as $\eps \to 0^+$, which ultimately implies that $\mathrm{div}\left(A\nabla v\right)=0$ in $B_{r_0}(x_0)$ and $v \in C^{1,\alpha}(B_{r_0}(x_0))$.
Finally, by \cite[Theorem 2.1]{Han} we know that $\mathrm{Cap}(S(u)\cap B_1) = 0$ and, by definition of $H^1$-capacity, there exists a sequence of cut-off functions $\eta_k \in C^\infty_c(B_1)$ such that $0\leq \eta_k\leq 1$, $\eta_k \to 0$ a.e. in $B_1$ and
$$
\eta_k\equiv 1 \mbox{ in }N_{1/k}(S(u))\cap B_1,\quad\norm{\nabla \eta_k}{L^2(B_1)}\to 0.
$$
Hence, given $\phi \in C^\infty_c(B_1)$, by testing the equation satisfied by $v$ in $B_1\setminus S(u)$ with $(1-\eta_k)\phi \in C^\infty_c(B_1\setminus S(u))$ we get
$$
\int_{B_1} (1-\eta_k) A\nabla v \cdot \nabla \phi \,\mathrm{d}x\leq \Lambda\norm{\phi}{L^\infty(B_1)}\norm{\nabla v}{L^2(B_1)}\norm{\nabla \eta_k}{L^2(B_1)}
$$
which leads to definition of weak solution in $B_1$, as $k\to +\infty$.
\end{proof}

\subsection{Sobolev inequalities and local boundedness of solutions}
In this section we provide Sobolev inequalities and local boundedness of solutions to \eqref{eqw} with constants which are uniform in the class $\mathcal S_{N_0}$.

\begin{Proposition}[Sobolev embedding]\label{Sobolev}
Given $u \in \mathcal{S}_{N_0}$ and $a\in (-a_{\mathcal{S}},a_{\mathcal{S}})\cup(1,+\infty)$, there exists $C,\varepsilon>0$, depending only on the class $\mathcal{S}_{N_0}$ and on $a$, such that the following inequality holds: for every $w \in H^1(B_1,|u|^a)$ and for any $0<r\leq\frac{7}{8}$, we have
\begin{equation}\label{sobolev.embed}
\left(\int_{B_r} |u|^a|w|^{2^*_{\mathcal{S}}(a)}\,\mathrm{d}x\right)^{2/2^*_{\mathcal{S}}(a)}\leq C \int_{B_{r}} |u|^a|\nabla w|^2\,\mathrm{d}x,
\end{equation}
where
$$
2^*_{\mathcal{S}}(a)= 2+ \eps.
$$
\end{Proposition}

\begin{proof}
We split the proof into two cases: $a\in (-a_{\mathcal{S}},a_{\mathcal{S}})$ and $a>1$.\\

Case 1: $a\in (-a_{\mathcal{S}},a_{\mathcal{S}})$. First, notice that Lemma \ref{l:a-N0} (ii) implies that $|u|^a$ is $\mathcal{A}_2$-Muckenhoupt in $B_{7/8}$. Then, it is possible to exploit \cite{Haj}, which provides weighted Sobolev embeddings with uniform constant in the class $\mathcal S_{N_0}$. The basic requirement for Sobolev inequalities for weighted Sobolev spaces involving general measures is a local growth condition on the measure of balls. Indeed, by following the notations of \cite{AlGoHaj, Haj,HajMEM}, the natural measure associated to \eqref{e:norm-ua} is defined as $d\mu = |u(x)|^a dx$. It is worth noting that $d\mu$ and the Lebesgue measure are reciprocally absolutely continuous. Indeed, on one hand, the local integrability of the weight $|u|^a$ implies that $d\mu$ is absolutely continuous with respect to the Lebesgue measure, while on the other hand, the vice versa follows from the fact that the nodal set $Z(u)$ has zero Lebesgue measure.

On the other hand, by rephrasing Lemma \ref{lemma.hay}, we get that $\mu$ is $(n+a^+ \overline N_0)$-regular in the sense of \cite{Haj}, and so
\begin{equation}\label{measure.regular}
\mu(B_r(x_0)) \geq C r^{n+a^+ \overline N_0} \mu (B_1),
\end{equation}
for every $x_0 \in B_{3/4}$ and $r\leq 7/8-|x_0|$
, where $C$ depends only on $\lambda,\Lambda,L$ and $N_0$ (see also \eqref{overlineN0} for the precise definition of $\overline{N}_0)$.

Then, the result follows by applying \cite[Theorem 5.1]{HajMEM} to the metric measure Sobolev space associated with the measure $d\mu = |u(x)|^a dx$ (see also \cite[Theorem 4]{AlGoHaj}). Moreover, in the present case, the expression of the critical exponent of the Sobolev embedding, i.e.
$$
2^*_{\mathcal{S}}(a)=\frac{2(n+a^+\overline N_0)}{n+a^+\overline N_0-2},
$$
follows by \eqref{measure.regular} and the dependence of the constant in \eqref{sobolev.embed} coincides with the one in \eqref{measure.regular}.\\

Case 2: $a>1$. In this setting, the strong degeneracy of the weight makes it possible to derive the Sobolev inequality \eqref{sobolev.embed} as a consequence of the classical case corresponding to $a=0$. 
Notice that this range cover the peculiar case $a=2$, which is our main interest here.

First, given $w \in H^1(B_1,|u|^a)$ and $r \in (0,1)$, by Lemma \ref{l:capacity} (i) and Proposition \ref{p:alternative-H} (i), we know that there exists $\phi_\eps \in C^\infty_c(B_r\setminus Z(u))$ approximating $w$  with respect to the norm \eqref{e:norm-ua}. Therefore,
by testing the equation $\mathrm{div}(A\nabla u)=0$ with the function $|u|^{a-1}\phi^2_\eps$ and integrating by parts in $B_r$, we obtain
$$
 \left(\frac{a-1}{2}\frac{\lambda}{\Lambda}\right)^2\int_{B_r}|u|^a\frac{|\nabla u|^2}{u^2}\phi_\eps^2\,\mathrm{d}x\leq\int_{B_r}|u|^a|\nabla \phi_\eps|^2\,\mathrm{d}x.
$$
Finally, by passing to the limit as $\eps \to 0^+$, we deduce the validity of the following Hardy-type inequality
\begin{equation}\label{hardy}
    \left(\frac{a-1}{2}\frac{\lambda}{\Lambda}\right)^2\int_{B_r}|u|^a\frac{|\nabla u|^2}{u^2}w^2\,\mathrm{d}x\leq\int_{B_r}|u|^a|\nabla w|^2\,\mathrm{d}x,\quad\mbox{for every }r\in (0,1).
\end{equation}
In particular, since 
$$
\int_{B_r}(|u|^{a/2} w)^2 + |\nabla (|u|^{a/2}w)|^2\,\mathrm{d}x \leq  C \int_{B_r} |u|^{a}( w^2 + |\nabla w|^2) + |u|^a \frac{|\nabla u|^2}{u^2}w^2\,\mathrm{d}x,
$$
the inequality \eqref{hardy} implies that $|u|^{a/2}w\in H^1(B_1)$ and 
\be\label{e:im}
\int_{B_r}|\nabla (|u|^{a/2}w)|^2\,\mathrm{d}x \leq c
\int_{B_r}|u|^a|\nabla w |^2\,\mathrm{d}x 
\ee
for some constant $c>0$ depending only on $n, \lambda/\Lambda$ and $a$. Thus, the classical Sobolev inequality ensures the existence of a dimensional constant $c>0$ such that, for every $0<r\leq1$, we have
\begin{equation}\label{sobolevstandard}
\left(\int_{B_r}||u|^{a/2}w|^{2^*}\,\mathrm{d}x\right)^{2/2^*}\leq c\int_{B_r}|\nabla(|u|^{a/2}w)|^2\,\mathrm{d}x,\qquad \mathrm{with \ }2^*:=\frac{2n}{n-2}.
\end{equation}
The conclusion follows once we  
prove the existence of two constants $\gamma\in(2,2^*)$ and $\overline C>0$, which depend only on $\mathcal S_{N_0}$ and $a$, such that the following inequality holds:
\begin{equation}\label{claimsobolev}
\left(\int_{B_r}|u|^a|w|^\gamma\,\mathrm{d}x\right)^{1/\gamma}\leq\overline{C}\left(\int_{B_r}||u|^{a/2}w|^{2^*}\,\mathrm{d}x\right)^{1/2^*}\quad\mbox{for every }r \in \left(0,\frac78\right].
\end{equation}
Indeed, the weighted Sobolev inequality \eqref{sobolev.embed} follows by combining \eqref{claimsobolev}, \eqref{e:im}, \eqref{sobolevstandard}, and \eqref{hardy}, with the choice $2^*_{\mathcal S}(a)=\gamma$.

To establish \eqref{claimsobolev}, we start by showing the existence of $\gamma \in (2,2^*), \overline{c}>0$, depending only on $\mathcal{S}_{N_0}$ and $a$, so that 
\be\label{e:12}
    \int_{B_{7/8}}|u|^{a\frac{2^*}{2}\frac{2-\gamma}{2^*-\gamma}}\,\mathrm{d}x\leq \overline c.
\ee
The latter uniform integrability condition in ensured by Lemma \ref{l:a-N0} (i), and works by choosing $\gamma \in (2,2^*)$ such that 
\begin{equation}\label{fraccondition}
a\frac{2^*}{2}\frac{\gamma-2}{2^*-\gamma}<a_{\mathcal S}, \quad\mbox{namely}\quad \gamma < \frac{2n(2+a\overline N_0)}{2n-4+an \overline N_0}.
\end{equation}
By direct computation, one can check that for $a\geq 1$ the right hand side in \eqref{fraccondition} is strictly between $2$ and $2^*$, and so  
\be\label{e:tuogamma}
\gamma = \frac12 \left(\frac{2n(2+a^+\overline N_0)}{2n-4+a n\overline N_0}+2\right)
\ee
is an admissible choice for \eqref{fraccondition}. Finally, given $\gamma$ as in \eqref{e:tuogamma}, consider the conjugate exponents $p:=2^*/\gamma$ and $q:=2^*/(2^*-\gamma)$. Then, by applying H\"older inequality and \eqref{e:12}, we get
\begin{align*}
\int_{B_r}|u|^a |w|^\gamma \,\mathrm{d}x &= \int_{B_r}||u|^{a/2} w|^\gamma |u|^{a(1-\frac{\gamma}{2})} \,\mathrm{d}x\\
&\leq C\left(\int_{B_r}||u|^{a/2} w|^{2^*}\,\mathrm{d}x\right)^{\frac{\gamma}{2^*}}
\left(\int_{B_r}|u|^{a\frac{2^*}{2}\frac{2-\gamma}{2^*-\gamma}}\,\mathrm{d}x\right)^{\frac{2^*-\gamma}{2^*}}\\
&\leq \overline{C}^\gamma\left(\int_{B_r}||u|^{a/2} w|^{2^*}\,\mathrm{d}x\right)^{\frac{\gamma}{2^*}},
\end{align*}
as we claimed.
\end{proof}

\begin{Proposition}[Local boundedness of energy solutions]\label{prop.bound}
Let $u \in \mathcal{S}_{N_0}, a \in (-a_{\mathcal{S}},a_{\mathcal{S}})\cup (1,+\infty)$ and $w\in H^1(B_1,|u|^a)$ be an energy solution to \eqref{eqw} in $B_1$. Then, for every $p\geq 2$ there exists $C>0$ depending only on the class $\mathcal{S}_{N_0}, a$ and $p$, such that, for $0< \rho < r<7/8$ it holds
\begin{equation}\label{loc.bound}
\|w\|_{L^\infty(B_{\rho})}\leq C \left(\frac{1}{(r-\rho)^{\alpha}}\int_{B_r} |u|^a w^p\,\mathrm{d}x\right)^{1/p},\quad\mbox{with }\,\alpha:= 2\frac{ 2^*_{\mathcal{S}}(a)}{2^*_{\mathcal{S}}(a)-2} .
\end{equation}
In particular, if $a \in (-a_{\mathcal{S}},a_{\mathcal{S}})$, we have $\alpha = n+a^+ \overline N_0$.
\end{Proposition}
\begin{proof}
We just sketch the ideas of the proof since it is based on the classical Moser iteration argument.\\
Let $\beta\geq 1, 0<\rho\leq r<7/8$ and $\eta \in C^\infty_c(B_r)$ be a non increasing radial cut-off function such that $\eta \in [0,1], \eta \equiv 1$ in $B_\rho$ and $|\nabla \eta|\leq 2(r-\rho)^{-1}$. Then, by testing \eqref{eqw} with $\eta^2 w^{\beta}$ and integrating by parts, we get the following Caccioppoli type inequality
$$
\int_{B_1} |u|^a|\nabla (\eta w^{\frac{1+\beta}{2}})|^2  \,\mathrm{d}x\leq C(\lambda,\Lambda)\left(\frac{1+\beta}{\beta}\right)^2\int_{B_1} |u|^a w^{1+\beta}|\nabla \eta|^2\,\mathrm{d}x.
$$
Then by the choice of $\eta$ and the Sobolev inequality \eqref{sobolev.embed}, we get
\begin{align*}
&\left(\int_{B_\rho} |u|^a w^{\frac{1+\beta}{2}2^*_{\mathcal{S}}(a)} \,\mathrm{d}x \right)^{\frac{2}{2^*_{\mathcal{S}}(a)}}  \leq \left(\int_{B_r} |u|^a \left(\eta w^{\frac{1+\beta}{2}}\right)^{2^*_{\mathcal{S}}(a)}\,\mathrm{d}x \right)^{\frac{2}{2^*_{\mathcal{S}}(a)}} \\ 
&\qquad\qquad\qquad\leq C(n,\lambda,\Lambda,N_0)\int_{B_r} |u|^a|\nabla (\eta w^{\frac{1+\beta}{2}})|^2 \,\mathrm{d}x\\
&\qquad\qquad\qquad\leq C(n,\lambda,\Lambda,\beta,N_0)
\frac{1}{(r-\rho)^2}\int_{B_r} |u|^a w^{1+\beta}\,\mathrm{d}x.
\end{align*}

 For the sake of readability, we omit the explicit dependence of the constant $C>0$, since it remains invariant through the rest of the proof. Thus, by setting $\nu= 2^*_{\mathcal{S}}(a)/2$ and $p=\beta+1$, we get
\begin{equation}\label{moser.ite}
\left(\int_{B_\rho} |u|^a w^{p\nu} \,\mathrm{d}x\right)^{\frac{1}{\nu}} \leq \frac{C}{(r-\rho)^2 }\int_{B_r} |u|^a w^{p}\,\mathrm{d}x,\quad\mbox{for every }p\geq 2.
\end{equation}
and, by applying repeatedly \eqref{moser.ite} to the sequence
$$
r_i := \rho + \frac{r-\rho}{2^i},\quad (r_i-r_{i+1})^2 = \frac{(r-\rho)^2}{4^{1+i}},\quad \sigma_{i}=p\nu^i
$$
we have
$$
\left(\int_{B_{r_{i+1}}} |u|^a w^{\sigma_{i+1}} \,\mathrm{d}x \right)^{\frac{1}{\nu^{i+1}}}\leq \left(\frac{C}{(r-\rho)^2}\right)^{\sum_{k=0}^i \nu^{-k}}\int_{B_r} |u|^a w^{p}\,\mathrm{d}x.
$$
Since
$$
\sum_{k=0}^\infty \frac{1}{\nu^k} = \frac{\nu}{\nu-1}=\frac{2^*_{\mathcal{S}}(a)}{2^*_{\mathcal{S}}(a)-2},
$$
if we set $\alpha$ as in \eqref{loc.bound}, we get that
$$
\left(\int_{B_{r_{i+1}}} |u|^a w^{\sigma_{i+1}} \,\mathrm{d}x\right)^{\frac{p}{\sigma_{i+1}}}\leq \frac{C}{(r-\rho)^{\alpha}}\int_{B_r} |u|^a w^{p}\,\mathrm{d}x,
$$
for every $p\geq 2$. Finally, by taking the limit as $i\to +\infty$, we infer that
$$
\|w\|_{L^\infty(B_{\rho},|u|^a)} \leq C(n,\lambda,\Lambda,N_0,p)\left(\frac{1}{(r-\rho)^{\alpha}}\int_{B_r} |u|^a w^{p}\,\mathrm{d}x\right)^{1/p},
$$
where the essential supremum is formulated in terms of the measure $d\mu = |u(x)|^a dx$. Then, \eqref{loc.bound} follows once we notice that since $Z(u)$ has null measure, then the Lebesgue measure is absolutely continuous with respect to $d\mu$.
\end{proof}

\begin{remark}\label{r:giorgiopoi}
We remark that for $a=2$, the Proposition \ref{prop.bound} implies an explicit bound of the ratio of two solutions sharing their zero sets, uniform-in-$\mathcal{S}_{N_0}$. More precisely, let $u\in \mathcal{S}_{N_0}$ and $v \in H^1(B_1)$ be solutions to \eqref{equv} such that $Z(u)\subseteq Z(v)$, then for $r \in (0,7/8)$
$$
\left\|\frac{v}{u}\right\|_{L^\infty(B_{r/2})}\leq C \left(\frac{1}{r^{\alpha}}\int_{B_r} v^2\,\mathrm{d}x\right)^{1/2},\quad\mbox{with }\,\alpha:= 2\frac{ 2^*_{\mathcal{S}}(2)}{2^*_{\mathcal{S}}(2)-2} 
$$
and $C>0$ depending only on $\mathcal{S}_{N_0}$. We refer to \cite[Theorem 4.2]{LinLin}, where the same result is deduced from a corkscrew property of nodal domains.
\end{remark}

\section{A priori uniform-in-\texorpdfstring{$\mathcal{S}_{N_0}$}{Lg} H\"{o}lder estimates for the ratio}\label{sec:holder}
In this section, we are going to prove Theorem \ref{uniformHolderZ}, that is, \emph{} uniform-in-$\mathcal S_{N_0}$ local $C^{0,\alpha}$ estimates in any dimension $n\geq2$ for ratios of solutions sharing nodal sets.

First, let us remark the following fact: \cite[Theorem 1.1]{LogMal1} proves real analitycity of the ratio $v/u$ in case the coefficients of the equation solved by $u$ and $v$ are real analytic, then in particular $v/u$ is $\alpha$-H\"older continuous for any $\alpha\in(0,1)$. Therefore, in case of real analytic coefficients, Theorem \ref{uniformHolderZ} immediately implies the following

\begin{Corollary}[Uniform-in-$\mathcal S_{N_0}$ H\"{o}lder bounds for the ratio]\label{corLogMal}
Let $n\geq2$, $A\in\mathcal A$ with real analytic entries, $\alpha\in(0,1)$, $u\in \mathcal{S}_{N_0}$ and $v$ be solutions to
$$
\mathrm{div}\left(A\nabla u\right) = \mathrm{div}\left(A\nabla v\right) = 0\quad\mbox{in }B_1,\quad\mbox{such that }Z(u)\subseteq Z(v).
$$
Then, there exists a positive constant $C$ depending only on $\mathcal{S}_{N_0}$ and $\alpha$ such that
\begin{equation*}
\left\|\frac{v}{u}\right\|_{C^{0,\alpha}(B_{1/2})}\leq  C\|v\|_{L^2(B_1)}.
\end{equation*}
\end{Corollary}

We would like to remark that the previous result extends to any dimension the $2$-dimensional uniform H\"older estimate implied by the uniform Harnack estimate in \cite[Theorem 1.1]{LogMal2}.

We would like to remark that, in view of Proposition \ref{p:inverse}, Theorem \ref{uniformHolderZ} is equivalent to the following result

\begin{Theorem}
Let $n\geq2, A\in\mathcal A, u \in \mathcal{S}_{N_0}$ and $w \in H^1(B_1,u^2)$ be a solutions to
$$
\mathrm{div}(u^2 A\nabla w) = 0\quad\mbox{in }B_1,
$$
in the sense of Definition \ref{definition.energy.a}. Then, if $w\in C^{0,\alpha}_{\loc}(B_1)$, there exists a positive constant $C$ depending only on $\mathcal{S}_{N_0}$ and $\alpha\in(0,1)$ such that
\begin{equation*}
\left\|w\right\|_{C^{0,\alpha}(B_{1/2})}\leq C\|w\|_{L^2(B_1,u^2)}.
\end{equation*}
\end{Theorem}

\subsection{Proof of Theorem \ref{uniformHolderZ}} 
The proof follows by combining the contradiction argument introduced in \cite{TerTorVit1} with the Liouville Theorem \ref{liouvilletheorem}.\\

Let us assume by contradiction, along a sequence $k\to+\infty$, the existence of $A_k\in \mathcal A$, $u_k\in \mathcal S_{N_0}$ and $v_k$ solutions to
$$
L_{A_k}u=L_{A_k}v=0,\qquad \mbox{with }Z(u_k)\subseteq Z(v_k)
$$
and $w_k:=v_k/u_k\in C^{0,\alpha}_{\loc}(B_1)$ for some $\alpha\in(0,1)$ and such that $\|w_k\|_{L^\infty(B_1)}=1$ and
$$\|w_k\|_{C^{0,\alpha}(B_{1/2})}\to +\infty.$$
Let us  consider a radially decreasing cut-off function $\eta\in C^\infty_c(B_1)$ with $0\leq\eta\leq1$, $\eta\equiv1$ in $B_{1/2}$ and such that $\mathrm{supp}\eta=B_{3/4}=:B$. Then, we can assume that for two sequences of points $x_k,\zeta_k\in B$, it holds
\begin{equation*}
\frac{|\eta w_k(x_k)-\eta w_k(\zeta_k)|}{|x_k-\zeta_k|^\alpha}=L_k\to+\infty.
\end{equation*}
Now we want to define two blow-up sequences. Set $r_k=|x_k-\zeta_k|\to0$ and consider
$$
\begin{aligned}
\tilde W_k(x)&=\frac{1}{L_kr_k^{\alpha}}\left((\eta w_k)(x_k+r_k A^{-1/2}_k(x_k)x)-(\eta w_k)(x_k)\right),\\ W_k(x)&=\frac{\eta(x_k)}{L_kr_k^{\alpha}}\left(w_k(x_k+r_k A^{-1/2}_k(x_k)x)-w_k(x_k)\right),
\end{aligned}
$$
for $x\in B_{\sqrt{\lambda}/4 r_k}$. Then $\tilde W_k$ are uniformly $\alpha$-H\"older continuous on compact sets of $\R^n$ and both $\tilde W_k,W_k$ converge uniformly on compact sets to an entire non constant profile $W$, which is globally $\alpha$-H\"older continuous in $\R^n$. On the other hand, by rescaling \eqref{eqw}, $W_k$ solves
\begin{equation*}
\mathrm{div}\left(U_k^2\tilde{A}_k\nabla W_k\right)=0\quad\mathrm{in} \ B_{\sqrt{\lambda}/4r_k},
\end{equation*}
in the sense of Definition \ref{definition.energy.a}, with
$$
\tilde{A}_k(x)=A^{-1/2}_k(x_k) A_k (x_k+ r_k A^{-1/2}_k(x_k)x) A^{-1/2}_k(x_k)\in \mathcal{A}
$$
and $U_k$ defined as in \eqref{e:blow-up}. By Proposition \ref{p:blow-up}, set $U$ to be the limit of the blow-up sequence $U_k$, where $Z(U_k)\to Z(U)$ locally with respect to the Hausdorff convergence. Moreover, since $U^2\in L^1_\loc(\R^n)$ and on every compact set of $\R^n$ the sequence $W_k$ is uniformly bounded, by passing to the limit a Caccioppoli inequality we deduce that $W$ has finite $H^1(U^2)$-norm on every compact set of $\R^n$ and so, in light of Proposition \ref{p:alternative-H}, $W_k \to W$ strongly in $H^1_\loc(\R^n,U^2)$. On the other hand, since $W_k$ and $\tilde{W}_k$ shadows each other, we deduce that $W_k\to W$ strongly in $C^{0,\alpha}_\loc(\R^n)$, where $W$ is a solution to
$$
\mathrm{div}\left(U^2\nabla W\right)=0\quad\mbox{in }\R^n.
$$
in the sense of Definition \ref{definition.energy.a}. Now, by Proposition \ref{lem.natural}, $U$ is an harmonic polynomial and, by Proposition \ref{p:inverse}, the function $V:=UW$ is entire harmonic with $Z(U)\subseteq Z(V)$. Finally,  the growth condition
$$|W(x)|\leq C\left(1+|x|\right)^\alpha\quad\mbox{in }\R^n,$$
together with the fact that $W$ is non constant, bring to a contradiction by applying the Liouville Theorem \ref{liouvilletheorem}.

\section{Liouville theorems}\label{sec:liouville}
This section is devoted to proving some Liouville theorems for entire harmonic functions and for ratios of entire harmonic functions sharing zero sets. The latter is equivalent to a Liouville theorem for entire solutions to \eqref{e:a-entire}. First, we recall a rigidity result for entire harmonic functions with bounded Almgren frequency at infinity.

\begin{Proposition}\label{lem.natural}
Let $u \in H^1_\loc(\R^n)$ be an entire harmonic function with bounded Almgren frequency at infinity; that is, there exists $N_\infty \in \R$ such that $\lim_{t \to +\infty}N(0,u,t):=N_\infty$. Then, $N_\infty\in\mathbb N$ and $u$ is an harmonic polynomial of degree $N_\infty$.
\end{Proposition}
 This result is commonly considered classic and is based on a straightforward blow-down analysis, as already stressed in the proof of \cite[Corollary 4.7]{LinLin}. Then, the proof is omitted.
\subsection{An algebraic formula for harmonic polynomials in two dimensions}
Let us recall the following two dimensional Liouville theorem by Nadirashvili \cite[Theorem 2]{Nad}, see also \cite[Corollary 1]{Lin2} for a similar result.
\begin{Proposition}[Nadirashvili's Liouville theorem]\label{Prop:Nad}
Let $u$ be entire harmonic in $\R^2$ with at most $k$ nodal regions. Then, $u$ is a harmonic polynomial of degree at most $k-1$.
\end{Proposition}
We show that Proposition \ref{lem.natural} and Proposition \ref{Prop:Nad} are equivalent in two dimensions. In fact, inspired by \cite{Nad}, by using the inverse of the stereographic projection into $\mathbb S^2$ and Euler's characteristic on $\mathbb S^2$ we prove the following characterization.
\begin{Proposition}\label{propNad}
Let $u$ be a non trivial entire harmonic function in $\R^2$. Then the following points are equivalent:
\begin{itemize}
\item[i)] there exists $k\in\mathbb N\setminus\{0\}$ such that $u$ has $k$ nodal regions;
\item[ii)] there exists $N_\infty=\lim_{t \to +\infty}N(0,u,t)\in \mathbb N$;
\end{itemize}
moreover the following algebraic formula holds true
\begin{equation}\label{formulaEul}
k=1+N_\infty+\sum_{z_i\in S(u)}(N(z_i,u,0^+)-1).
\end{equation}
\end{Proposition}
\proof
If $k=1$, then by the classical Liouville Theorem $u$ is constant, and hence $N_\infty=0$. On the other hand, if $N_\infty=0$ then also in this case $u$ is constant and so $k=1$. Moreover $Z(u)=\emptyset$ so \eqref{formulaEul} trivially holds. Let us now prove the result in the nontrivial cases.

The fact that $i)$ implies $ii)$ follows by \cite[Theorem 2]{Nad}. In fact if $u$ is entire harmonic in $\R^2$ and satisfies $i)$ then it is a harmonic polynomial and $N_\infty$ exists finite and it coincides with the degree of $u$. In order to prove that $ii)$ implies $i)$ and formula \eqref{formulaEul} let us construct the following model. First, we perform the inverse of the stereographic projection from the plane into the $2$-sphere $\mathbb S^2$ by identifying the point at ${\infty}$ with the north pole. Then, the image of the nodal set $Z(u)$ in $\mathbb S^2$ defines a cell complex with a finite collection of $N_1$ smooth-arcs intersecting at $N_0$ endpoints and with $N_2$ connected components. Hence, the Euler characteristic for a cell complex gives
\begin{equation}\label{Eul}
N_2 = \chi(\mathbb S^2) - N_0 + N_1 = 2-N_0 + N_1,
\end{equation}
where $N_i$ denotes the number of $i$-dimensional cells of $\mathbb S^2$. Notice that $N_2$ coincides with the number of connected components of $\R^2\setminus Z(u)$; that is $N_2=k$.\\
First, in order to avoid the occurrence of loops, we add a vertex at every smooth arc connecting the north pole to itself. Indeed, by identifying the north pole with ${\infty}$, it is straightforward to deduce that those arcs coincide with the image of arcs of $Z(u)$ in which no singular points occur and with endpoints at infinity. Through this proof we will denote this collection of vertices as \emph{regular vertices} $RV(u)$ (see Figure \ref{figure1}). Thus, the collection of $N_0$ vertices consists in the image of singular points of $Z(u)$, the north pole and $RV(u)$ and so
$$
N_0 = 1+ \#S(u) + \#RV(u).
$$
In order to count the number of edges $N_1$ one has to remark that any branch of the nodal set $Z(u)$ either ends at ${\infty}$ or connects two singular points, otherwise by maximum principle the harmonic function $u$ would be identically zero on an open non-empty set, in contradiction with the unique continuation principle. Hence, by denoting with $C(u)$ the finite-set of branches connecting singular points, we get
\begin{equation*}
N_1= 2N_\infty + \# C(u).
\end{equation*}

\tikzset{every picture/.style={line width=0.75pt}}     
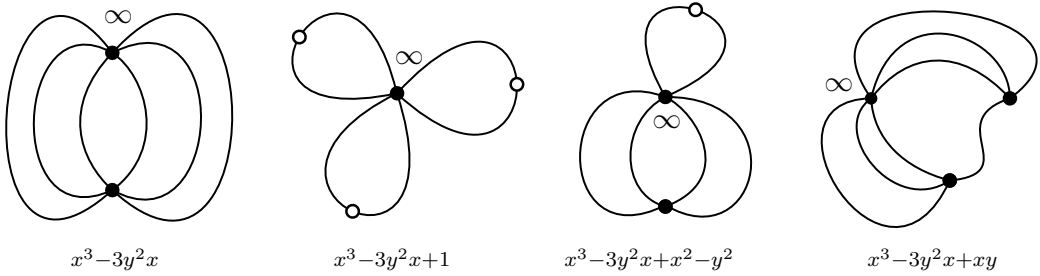
\begin{figure}
\begin{tikzpicture}[x=0.75pt,y=0.75pt,yscale=-0.7,xscale=0.7]
%uncomment if require: \path (0,217); %set diagram left start at 0, and has height of 217

%Curve Lines [id:da3302498101501209] 
\draw    (106.17,155) .. controls (146.17,125) and (131.22,72.72) .. (106.17,56) ;
\draw [shift={(106.17,56)}, rotate = 213.72] [color={rgb, 255:red, 0; green, 0; blue, 0 }  ][fill={rgb, 255:red, 0; green, 0; blue, 0 }  ][line width=0.75]      (0, 0) circle [x radius= 4.5, y radius= 4.5]   ;
\draw [shift={(106.17,155)}, rotate = 323.13] [color={rgb, 255:red, 0; green, 0; blue, 0 }  ][fill={rgb, 255:red, 0; green, 0; blue, 0 }  ][line width=0.75]      (0, 0) circle [x radius= 4.5, y radius= 4.5]   ;
%Curve Lines [id:da756247340550776] 
\draw    (106.17,155) .. controls (82.33,133.17) and (69,94.94) .. (106.17,56) ;
%Curve Lines [id:da14365774347745786] 
\draw    (106.17,56) .. controls (18,-62.5) and (-8.67,259.5) .. (106.17,155) ;
%Curve Lines [id:da9260335537823299] 
\draw    (106.17,155) .. controls (187.33,199.5) and (195.33,10.5) .. (106.17,56) ;
%Curve Lines [id:da8665130520074756] 
\draw    (106.17,155) .. controls (224,262.83) and (218.67,-67.17) .. (106.17,56) ;
%Curve Lines [id:da28079998003666273] 
\draw    (106.17,155) .. controls (26.33,191.5) and (35.33,16.17) .. (106.17,56) ;
%Curve Lines [id:da4107463430951215] 
\draw    (652.89,88.83) .. controls (670.44,12.5) and (752.67,42.28) .. (752.89,88.83) ;
\draw [shift={(752.89,88.83)}, rotate = 89.73] [color={rgb, 255:red, 0; green, 0; blue, 0 }  ][fill={rgb, 255:red, 0; green, 0; blue, 0 }  ][line width=0.75]      (0, 0) circle [x radius= 4.5, y radius= 4.5]   ;
\fill [black, shift={(652.89,88.83)}, rotate = 282.95] (0, 0) circle [x radius= 4.5, y radius= 4.5]   ;
%Curve Lines [id:da7987752869945599] 
\draw    (709.56,148.06) .. controls (764.33,142.83) and (702.33,105.5) .. (752.89,88.83) ;
\draw [shift={(709.56,148.06)}, rotate = 354.55] [color={rgb, 255:red, 0; green, 0; blue, 0 }  ][fill={rgb, 255:red, 0; green, 0; blue, 0 }  ][line width=0.75]      (0, 0) circle [x radius= 4.5, y radius= 4.5]   ;
%Curve Lines [id:da8951090450932145] 
\draw    (709.56,148.06) .. controls (675.78,139.17) and (654.44,117.83) .. (652.89,88.83) ;
\draw [shift={(709.56,148.06)}, rotate = 194.74] [color={rgb, 255:red, 0; green, 0; blue, 0 }  ][fill={rgb, 255:red, 0; green, 0; blue, 0 }  ][line width=0.75]      (0, 0) circle [x radius= 4.5, y radius= 4.5]   ;
%Curve Lines [id:da025181393410939834] 
\draw    (652.89,88.83) .. controls (620.67,112.06) and (669.56,174.72) .. (709.56,148.06) ;
%Curve Lines [id:da6261741525290059] 
\draw    (652.89,88.83) .. controls (564.22,90.28) and (663.33,248.94) .. (709.56,148.06) ;
%Curve Lines [id:da7600338694034906] 
\draw    (652.89,88.83) .. controls (678.89,54.72) and (717.11,50.28) .. (752.89,88.83) ;
%Curve Lines [id:da7432641153317667] 
\draw    (652.89,88.83) .. controls (583.78,-10.61) and (842.89,22.72) .. (752.89,88.83) ;
%Curve Lines [id:da07153479820360908] 
\draw    (310.83,84) .. controls (277,-58.17) and (160.33,117.17) .. (311.33,85.17) ;
\draw [shift={(311.33,85.17)}, rotate = 348.03] [color={rgb, 255:red, 0; green, 0; blue, 0 }  ][fill={rgb, 255:red, 0; green, 0; blue, 0 }  ][line width=0.75]      (0, 0) circle [x radius= 4.5, y radius= 4.5]   ;
\draw [shift={(240.93,44.61)}, rotate = 116.86, draw=black, fill=white, line width=0.95]      (0, 0) circle [x radius= 4.5, y radius= 4.5]   ;
%Curve Lines [id:da731095086994859] 
\draw    (311.33,85.5) .. controls (429.17,193.33) and (423.83,-38) .. (311.33,85.17) ;
\draw [shift={(397.68,78.72)}, rotate=263.51, draw=black, fill=white, line width=0.95] (0,0) circle [x radius=4.5, y radius=4.5];
%Curve Lines [id:da8012918875793007] 
\draw    (311.33,85.17) .. controls (173,153.83) and (362.33,242.5) .. (311.33,85.5) ;
\draw [shift={(279.3,170.28)}, rotate = 18.48, draw=black, fill=white, line width=0.95]      (0, 0) circle [x radius= 4.5, y radius= 4.5]   ;
%Curve Lines [id:da4478819596531999] 
\draw    (504.67,166.83) .. controls (555.78,147.39) and (528.22,99.39) .. (504.67,87.83) ;
\draw [shift={(504.67,166.83)}, rotate = 339.17] [color={rgb, 255:red, 0; green, 0; blue, 0 }  ][fill={rgb, 255:red, 0; green, 0; blue, 0 }  ][line width=0.75]      (0, 0) circle [x radius= 4.5, y radius= 4.5]   ;
%Curve Lines [id:da5473652847014727] 
\draw    (504.67,166.83) .. controls (470.89,157.94) and (471.33,105.06) .. (504.67,87.83) ;
\draw [shift={(504.67,166.83)}, rotate = 194.74] [color={rgb, 255:red, 0; green, 0; blue, 0 }  ][fill={rgb, 255:red, 0; green, 0; blue, 0 }  ][line width=0.75]      (0, 0) circle [x radius= 4.5, y radius= 4.5]   ;
%Curve Lines [id:da02738415894168189] 
\draw    (504.67,87.83) .. controls (600.22,87.28) and (572.67,205.94) .. (504.67,165.83) ;
%Curve Lines [id:da34887967372728457] 
\draw    (504.67,87.83) .. controls (422.89,57.5) and (422.44,207.28) .. (504.67,165.83) ;
%Curve Lines [id:da7750283668480298] 
\draw    (504.67,87.83) .. controls (438.89,-26.28) and (628.22,34.39) .. (505.33,87.28) ;
\draw [shift={(526.42,24.93)}, rotate = 21.64, draw=black, fill=white, line width=0.95]      (0, 0) circle [x radius= 4.5, y radius= 4.5]   ;
\draw [shift={(504.67,87.83)}, rotate = 240.04] [color={rgb, 255:red, 0; green, 0; blue, 0 }  ][fill={rgb, 255:red, 0; green, 0; blue, 0 }  ][line width=0.75]      (0, 0) circle [x radius= 4.5, y radius= 4.5]   ;

\draw (98,23.57) node [anchor=north west][inner sep=0.75pt]    {$\infty $};
\draw (108,205.01) node [inner sep=0.75pt]  {${\scriptstyle x^3-3y^2 x}$};
\draw (617.11,72.46) node [anchor=north west][inner sep=0.75pt]    {$\infty $};
\draw (697.11,205.01) node [inner sep=0.75pt]  {${\scriptstyle x^3-3y^2 x+xy}$};
\draw (307.67,53.23) node [anchor=north west][inner sep=0.75pt]    {$\infty $};
\draw (307.67,205.01) node [inner sep=0.75pt]  {${\scriptstyle x^3-3y^2 x+ 1}$};
\draw (493.11,99.01) node [anchor=north west][inner sep=0.75pt]    {$\infty $};
\draw (493.11,205.01) node [inner sep=0.75pt]  {$ {\scriptstyle x^3-3y^2 x+x^2-y^2}$};
\end{tikzpicture}
\vspace{-1.2cm}\\
\caption{In this figure, we illustrate different scenarios related to harmonic polynomials of degree $N_\infty = 3$. Specifically, the black dots represent the singular points $S(u)$, the white dots correspond to $RV(u)$, and the point $\infty$ denotes the north pole on $\mathbb{S}^2$.}
\label{figure1}
       \end{figure}

Thus, by collecting the previous equalities, equation \eqref{Eul} translates into
\begin{equation}\label{bound}
k = 1 - \#S(u) - \#RV(u) + 2N_\infty + \# C(u).
\end{equation}
Moreover, by combining the maximum principle with the unique continuation principle, we deduce that two possibilities arise: either $\#S(u)=0$ and so $\#RV(u)\geq 1, \#C(u)=0$; or $\#S(u)\geq 1$ and the number of connections is strictly smaller than the number of singular points; that is, $\#C(u) \leq \#S(u)-1$. Therefore, in both cases, we get
\begin{equation}\label{1}
k \leq 2N_\infty
\end{equation}
which ensures that $ii)$ implies $i)$.
Let us conclude by giving an alternative expression of $N_1$.
By the classical blow-up analysis for harmonic functions, at every point $z_i \in S(u)$ there exists $2N(z_i,u,0^+)$ smooth arcs intersecting at $z_i$ and connecting $z_i$ either to the north pole or to another singular point. On the other hand, by the blow-down analysis of Proposition \ref{lem.natural}, we know that there exist $2N_\infty$ arcs intersecting at the north pole. Thus, since every edge in $C(u)$ connect two singular points, we get
\begin{equation}\label{2}
2N_\infty = 2\#RV(u) + 2\sum_{z_i \in S(u)}N(z_i,u,0^+) - 2\#C(u)
\end{equation}
where we have used that every branch with endpoint at the north pole either connect with a regular vertex or with a singular point. Finally, by combining \eqref{bound} with \eqref{2}, we deduce the validity of equation \eqref{formulaEul}.
\endproof

\begin{remark}
Fixed $N_\infty\in\mathbb N\setminus\{0\}$, from \eqref{formulaEul}, the fact that $\sum_{z_i\in S(u)}(N(z_i,u,0^+)-1)\geq0$ and \eqref{1} it follows
$$
\frac{k}{2}\leq N_\infty\leq k-1
$$
and the two equalities holds at the same time if and only if $u$ is a linear function. Moreover, the two bounds are sharp for any choice of $N_\infty\in\mathbb N\setminus\{0\}$. In fact, any $N_\infty$-homogeneous harmonic polynomial has exactly $k=2N_\infty$ nodal regions. Then, considering $u$ a harmonic polynomial of degree $N_\infty$, by Sard's theorem there exists $c\in\R$ such that $u+c$, which is still a harmonic polynomial of degree $N_\infty$, is such that $S(u+c)=\emptyset$ and hence \eqref{formulaEul} says that the number of its nodal domains is $k=N_{\infty}+1$.
\end{remark}

\begin{remark}
Notice that if either $i)$ or $ii)$ holds true, then \eqref{formulaEul} implies
$$
0\leq\#S(u)\leq k-1-N_\infty.
$$
\end{remark}
\subsection{A Liouville theorem for the ratio of entire harmonic functions sharing zero sets}
The aim of this last section is to prove Theorem \ref{liouvilletheorem}. We will show that actually the latter result follows by the division lemma in \cite{Mur} which can be stated in the following version
\begin{Lemma}[Murdoch's division lemma]\label{gen.division.lemma}
Let $Q$ be a harmonic polynomial and $P$ be a polynomial such that $Z(Q)\subseteq Z(P)$. Then, there exists a polynomial $R$ such that $P=QR$.
\end{Lemma}
Actually the division lemma in \cite{Mur} is stated in a more general setting, and the result above follows by combining Theorem 2 and Lemma 4 in \cite{Mur}. Nevertheless,  one can relax the hypotheses on $Q$ by requiring that only its higher order term is harmonic.
\begin{proof}[Proof of Theorem \ref{liouvilletheorem}]
  Let $\mathrm{deg}(u)=N_\infty$. By \eqref{growth.eq} we get
  $$
  |v(x)|\leq C|u(x)|(1+|x|)^{\gamma}\leq C(1+|x|)^{\gamma + N_\infty}\quad\mbox{in }\R^n.
  $$
  Then $v$ is an harmonic polynomial and $\mathrm{deg}(v)\leq N_\infty\! +\! \lfloor \gamma \rfloor\!$, thus the result follows by Lemma \ref{gen.division.lemma}.
\end{proof}

\end{document}